\documentclass[10pt,a4paper,twoside]{amsart}
\usepackage{graphicx,enumitem}
\usepackage{a4wide}
\usepackage{color}
\usepackage{wasysym}
\usepackage{url}
\usepackage[colorlinks=true]{hyperref}
\usepackage{upref}
\usepackage{caption}
\usepackage{mathtools}

\usepackage{latexsym}

\allowdisplaybreaks
\numberwithin{equation}{section}

\usepackage{cmdtrack}

\newcommand{\IE}{\mathbb{E}}
\newcommand{\IR}{\mathbb{R}}
\newcommand{\IN}{\mathbb{N}}
\newcommand{\IP}{\mathbb{P}}
\newcommand{\cB}{{\mathcal B}}
\newcommand{\cE}{{\mathcal E}}
\newcommand{\cF}{{\mathcal F}}
\newcommand{\cG}{{\mathcal G}}
\newcommand{\cL}{{\mathcal L}}
\newcommand{\cM}{{\mathcal M}}

\newcommand{\Dx}{{\Delta x}}
\newcommand{\Dt}{{\Delta t}}
\newcommand{\st}{\;\bigm|\;}
\newcommand{\test}{\varphi}
\newcommand{\R}{\mathbb{R}}
\newcommand{\n}{\noindent}
\newcommand{\dis}{\displaystyle}
\newcommand{\ov}{\overline}
\newcommand{\om}{\omega\notag}
\newcommand{\rt}{\rightarrow}
\newcommand{\wh}{\widehat}
\newcommand{\norm}[1]{\left\|#1\right\|}
\newcommand{\abs}[1]{\left|#1\right|}

\newcommand{\Pt}{\Pi_T}

\newcommand{\seq}[1]{\left\{#1\right\}}
\newcommand{\Dm}{D_-}
\newcommand{\Dp}{D_+}
\newcommand{\CutAf}{C_{u_0,t,A,f}}
\newcommand{\whCutAf}{\widehat{C}_{u_0,t,A,f}}
\DeclareMathOperator*{\dv}{div}
\DeclareMathOperator*{\grad}{grad}

\newtheorem{lemma}{Lemma}[section]
\newtheorem{proposition}[lemma]{Proposition}

\newtheorem{theorem}[lemma]{Theorem}
\newtheorem{remark}[lemma]{Remark}

\newtheorem{definition}[lemma]{Definition}

\DeclareMathOperator*{\Lip}{Lip}

\DeclareMathOperator*{\essup}{ess\;sup}

\newcommand{\cs}[1]{#1}

\numberwithin{equation}{section}
 {\ifnum\value{asnr}=0 \stepcounter{asnr} 
   \begin{enumerate}[label=\textbf{A}.\arabic{enumi}]
     \else
     \begin{enumerate}[label=\textbf{A}.\arabic{enumi},resume] \fi}
 {\end{enumerate}}
 \newcounter{defnr}
 \newenvironment{Definitions} %
 {\ifnum\value{defnr}=0 \stepcounter{defnr} 
   \begin{enumerate}[label=\textbf{D}.\arabic{enumi}]
     \else
     \begin{enumerate}[label=\textbf{D}.\arabic{enumi},resume] \fi}
 {\end{enumerate}}

\graphicspath{{figures/}}

\begin{document}
\title[MLMC diffusion]{ Multilevel Monte Carlo for random degenerate scalar convection diffusion equation}

\author[U. Koley]{U. Koley} \address[Ujjwal Koley] {\newline  
   Institut f\"{u}r Mathematik,  \newline
   Julius-Maximilians-Universit\"{a}t W\"{u}rzburg,   
\newline Campus Hubland Nord, Emil-Fischer-Strasse 30, \newline 97074,
W\"{u}rzburg, Germany.} 
\email[]{toujjwal@gmail.com}

\author[N. H. Risebro]{N. H. Risebro} \address[Nils Henrik
Risebro]{\newline Centre of Mathematics for Applications (CMA) 
  \newline University of Oslo\newline P.O. Box 1053, Blindern\newline
 N--0316 Oslo, Norway} \email[]{nilshr@math.uio.no}

\author[Ch.~Schwab]{Ch. Schwab$^+$}\thanks{$^+$Research supported in part by ERC AdG247277 STAHDPDE} 
\address[Christoph Schwab]{
\newline Seminar for Applied Mathematics \newline
   ETH\newline HG G. 57.1, \newline R\"amistrasse 101, Z\"urich,
   Switzerland.} \email{schwab@sam.math.ethz.ch}

\author[F. Weber]{F. Weber}
\address[Franziska Weber] {\newline Centre of Mathematics for Applications (CMA)
  \newline University of Oslo\newline P.O. Box 1053, Blindern\newline
  N--0316 Oslo, Norway}\email{franziska.weber@cma.uio.no}

 \date{\today} 

\begin{abstract}
\cs{
We consider the numerical solution of 
scalar, nonlinear degenerate
convection-diffusion problems
with random diffusion coefficient and with
random flux functions. 
Building on recent results on 
the existence, uniqueness and 
continuous dependence of weak solutions
on data in the deterministic case, 
we develop a definition of 
random entropy solution.
We establish
existence, uniqueness, measurability and 
integrability results for these random 
entropy solutions, 
generalizing \cite{Mishr478,MishSch10a}
to possibly degenerate 
hyperbolic-parabolic problems with random data.
We next address the numerical approximation
of random entropy solutions,
specifically the approximation of the deterministic
first and second order statistics.
To this end, we consider explicit and implicit time discretization
and Finite Difference methods in space, and single as well as
Multi-Level Monte-Carlo methods to sample the statistics.
We establish convergence rate estimates with respect to
the discretization parameters, as well as with respect to the
overall work, indicating substantial gains in efficiency
are afforded under realistic regularity assumptions by the use
of the Multi-Level Monte-Carlo method.
Numerical experiments are presented which confirm the theoretical
convergence estimates.
}
\end{abstract}

\maketitle

\tableofcontents

\section{Introduction}
\label{sec:intro}
Many problems in physics and engineering are modeled by nonlinear, 
possibly strongly degenerate, convection diffusion equation.  
The Cauchy problem for such equations takes the form
\begin{equation}
  \label{eq:main1}
  \begin{cases}
    u_t + \dv f(u) = \Delta A(u), &(x,t)\in \Pt,\\
    u(0,x)=u_0(x), & x\in \R^d,
  \end{cases}
\end{equation}
where $\Pt=\R^d\times (0,T)$ with $T>0$ fixed, $u:\Pt\to \R$ is the
unknown function, $f=(f_1,\ldots,f_d)$ the flux function, and $A$ the nonlinear
diffusion. Regarding this, the basic assumption is that
$a(u):=A'(u)\ge 0$ for all $u$.  When \eqref{eq:main1} is
nondegenerate, i.e., $a(u) > 0$, it is well known that
\eqref{eq:main1} admits a unique classical solution \cite{ok1961}. This
contrasts with the degenerate case where $a(u)$ may vanish for some
values of $u$. A simple example of a degenerate equation is the porous
medium equation
\begin{align*}
 u_t = \Delta (u^m), \  m>1, 
\end{align*}
which degenerates at $u = 0$. 
This equation has served as a simple model to describe processes 
involving fluid flow, heat transfer or diffusion. 
Examples of applications are in the description of the flow of an 
isentropic gas through a porous medium, modelled by Leibenzon \cite{leibenzon1930} and 
Muskat \cite{muskat1936} around 1930, in the study of groundwater flow by 
Boussisnesq in 1903 \cite{boussinesq} or in heat radiation in plasmas, 
Zel'dovich and collaborators around 1950, \cite{zeldovich}. 
In general, a manifestation of the
degeneracy in \eqref{eq:main1} is the finite speed of propagation of
disturbances. If $a(0)=0$, and if at some fixed time the solution $u$ has compact
support, then it will continue to have compact support for all later
times.

By the term ``strongly degenerate'' we mean that there 
is an open interval such that $a(u)=0$ if $u$ is in this interval.
Hence, the class of equations under consideration is very
large and contains the heat equation, the porous
medium equation and scalar conservation laws.
Independently of the smoothness of the initial data, due to the
degeneracy of the diffusion, singularities may form in the solution
$u$. Therefore we consider \cs{weak solutions which are defined
as follows}.
\begin{definition}
  Set $\Pt=\R^d \times (0,T)$. 
  A function 
  \begin{equation*}
    u(t,x) \in C\left([0,T];L^1(\R^d)\right) \cap L^\infty(\Pt)
  \end{equation*}
  is a weak solution of the initial value problem 
  \eqref{eq:main1} if it satisfies:
  \begin{Definitions}
  \item $\grad A(u) \in L^\infty(\Pt)$.\label{def:w1}
  \item For all test functions $\test\in
    \mathcal{D}(\R^d \times [0,T))$ \label{def:w2}
    \begin{equation}
      \label{eq:weaksol}
      \iint_{\Pt} \left( u\test_t + f(u)\cdot \grad \test + A(u)
        \Delta \test \right)\,dx \,dt 
      + \int_{\R^d} u_0(x) \test(x,0) \,dx= 0.
    \end{equation}
  \end{Definitions}
\end{definition}
In view of the existence theory, the condition~\ref{def:w1} is
natural, and thanks to this we can replace \eqref{eq:weaksol} 
by
\begin{equation*}
  \iint_{\Pt} u\test_t +\left( f(u) - \grad A(u)\right)\cdot \grad\test \,dxdt +
  \int_{\R^d} u_0(x) \test(x,0) \,dx= 0. 
\end{equation*}
If $A$ is constant on a whole interval, then weak solutions are not
uniquely determined by their initial data, and one must impose an
additional entropy condition to single out the physically relevant
solution. A weak solution satisfies the entropy condition if
\begin{equation}
  \label{eq:entcond1}
  \varrho(u)_t + \dv q(u) - \Delta r(u) \le 0\ \text{in $\mathcal{D}'(\Pt)$,}
\end{equation}
for all convex, twice differentiable functions $\varrho:\R\to \R$,
where $q$ and $r$ are defined by
\begin{equation*}
  q'(u)=\varrho'(u)f'(u), \ \text{and}\ r'(u)=\varrho'(u) A'(u).
\end{equation*}
Via a standard limiting argument this implies that \eqref{eq:entcond1}
holds for the Kru\v{z}kov entropies $\varrho(u)=\abs{u-c}$ for all
constants $c$. We call a weak solution satisfying the entropy
condition an entropy solution.

For scalar conservation laws, the entropy framework (usually called
entropy conditions) was introduced by Kru\v{z}kov \cite{Kruzkov} and
Vol'pert \cite{Volpert}, while for degenerate parabolic equations
entropy solution were first considered by Vol'pert and Hudajev
\cite{VolpertHudajev}. 
Uniqueness of entropy solutions to \eqref{eq:main1} 
was first proved by Carrillo \cite{Carrillo}. 

Over the years, there has been a growing interest in numerical
approximation of entropy solutions to degenerate parabolic
equations. Finite difference and finite volume schemes for degenerate
equations were analysed by Evje and Karlsen
\cite{ek2000,EvjeKarlsen2, ek1998, EvjeKarlsen4} (using
upwind difference schemes), Holden \textit{et al}.~\cite{Holdenetal1,
  Holdenetal2} (using operator splitting methods), Kurganov and Tadmor
\cite{KurganovTadmor} (central difference schemes), Bouchut \textit{et
  al}.~\cite{Boucutetal} (kinetic BGK schemes), Afif and Amaziane
\cite{AfifAmaziane} and Ohlberger, Gallou\"{e}t \textit{et
  al}.~\cite{Ohlberger, Gall-1, Gall-2} (finite volume methods),
Cockburn and Shu \cite{CockburnShu} (discontinuous Galerkin methods)
and Karlsen and Risebro \cite{KRT1,KRT2} (monotone difference
schemes). Many of the above papers show that the approximate solutions
converge to the unique entropy solution as the discretization
parameter vanishes. Rigorous estimates of the convergence rate of
finite volume schemes for degenerate parabolic equations were proved
in \cite{KarlsenRisebroStorrosten1} (1-d) and
\cite{KarlsenRisebroStorrosten2} (multi-d).

This \emph{classical} paradigm for designing efficient numerical
schemes assumes that {\em data for \eqref{eq:main1}, i.e.,
  initial data $u_0$, convective flux and diffusive flux are known exactly.  }

In many situations of practical interest, however, these data are not
known exactly due to inherent uncertainty in modelling and
measurements of physical parameters such as, for example, the specific
heats in the equation of state for compressible gases, or the relative
permeabilities in models of multi-phase flow in porous media.  Often,
the initial data are known only up to certain statistical quantities
of interest like the mean, variance, higher moments, and in some
cases, the law of the stochastic initial data. In such cases, a
mathematical formulation of \eqref{eq:main1} is required which allows
for \emph{random data}. The problem of random initial data was
considered in \cite{MishSch10a}, and the existence and uniqueness of a
random entropy solution was shown, and a convergence analysis for MLMC
FV discretizations was given. \cs{The} MLMC discretization of balance
laws with random source terms was investigated in \cite{MishSchSuk11b}.

In \cite{MishSch10a} a mathematical framework was developed for
scalar conservation laws with random initial data. This framework was
extended to include random flux functions in \cite{Mishr478}.

The aim of this paper is \emph{ to extend this 
  mathematical framework to include degenerate
  convection diffusion equations with random convective and diffusive
  flux functions with \cs{possibly} correlated random perturbations.}
\cs{
Its outline is as follows. In Section \ref{sec:prob} we review notions
from probability and from random variables taking values in separable
Banach spaces. Section \ref{ssec:SCL} is devoted to a review of 
convergence rates from \cite{KarlsenRisebroStorrosten1,KarlsenRisebroStorrosten2}
on convergence rates for scalar, degenerate deterministic 
convection-diffusion problems. Particular attention is paid to
the definition of entropy solutions and to existence-, uniqueness-
and continuous dependence results, and to the definition of the 
random entropy solutions, and to sufficient conditions ensuring 
their measurability and integrability.
In Section \ref{sec:schemes}, we then address the discretization.
First, again reviewing convergence rates of FD schemes for the 
deterministic case from \cite{KarlsenRisebroStorrosten1,KarlsenRisebroStorrosten2},
which we then extend to Monte-Carlo as well as Multi-Level Monte-Carlo
versions for the degenerate convection-diffusion problem with
random coefficients and flux functions.
The final Section \ref{sec:NumExp} is then devoted to numerical
experiments which confirm the theoretical convergence estimates
and, in fact, indicate that they probably are pessimistic, at 
least in the particular test problems considered.
}

\section{Preliminaries from Probability}
\label{sec:prob}
We use the concept of random variables taking values in function
spaces.  To this end, we recapitulate basic concepts from
\cite[Chapter 1]{DaPrZab91}.

Let $(\Omega, \cF)$ be a measurable space, with $\Omega$ denoting the
set of all elementary events, and $\cF$ a $\sigma$-algebra of all
possible events in our probability model. If $(E, \cG)$ denotes a
second measurable space, then an {\em $E$-valued random variable} (or
random variable taking values in $E$) is any mapping $X: \Omega
\rightarrow E$ such that the set $\{\omega \in \Omega$: $X(\omega) \in
A\} = \{X \in A\} \in \cF$ for any $A \in \cG$, i.e., such that $X$ is
a $\cG$-measurable mapping from $\Omega$ into $E$.

Assume now that $E$ is a metric space; with the Borel $\sigma$-field
$\cB(E)$, $(E, \cB(E))$ is a measurable space and we shall always
assume that $E$-valued random variables $X: \Omega \rightarrow E$ will
be $(\cF, \cB(E))$ measurable.  If $E$ is a separable Banach space
with norm $\|\circ \|_E$ and (topological) dual $E^*$, then $\cB(E)$
is the smallest $\sigma$-field of subsets of $E$ containing all sets
\begin{equation*}
  \{x \in E: \varphi(x) \le \alpha\}, \; \varphi \in E^*, \;\alpha \in \IR\,.
\end{equation*}

\n Hence if $E$ is a separable Banach space, $X: \Omega \rightarrow E$
is an $E$-valued random variable iff for every $\varphi \in E^*$,
$\omega \longmapsto \varphi (X(\omega)) \in \IR^1$ is an
$\IR^1$-valued random variable. 
Moreover, by \cite[Lemma 1.5, p.19]{DaPrZab91} 
the norm $\Omega\ni \omega\mapsto\|X(\omega)\|_E\in \mathbb{R}$ is a measurable mapping.




%
%


The random variable $X: \Omega \rightarrow E$ is called {\em Bochner
  integrable} if, for any probability measure $\IP$ on the measurable
space $(\Omega, \cF)$,
\begin{equation*}
  \dis\int_\Omega \|X(\om)\|_E \, \IP(d\om) < \infty\,.
\end{equation*}

\medskip\n A probability measure $\IP$ on $(\Omega, \cF)$ is any
$\sigma$-additive set function from $\Omega$ into $[0,1]$ such that
$\IP(\Omega) = 1$, and the measure space $(\Omega, \cF, \IP)$ is
called probability space.  We shall assume that $(\Omega, \cF, \IP)$
is complete.

If $X: (\Omega, \cF) \rightarrow (E, \cE)$ is a random variable,
$\cL(X)$ denotes the {\em law of $X$ under $\IP$}, i.e.,
\begin{equation*}
  \cL(X) (A) = \IP (\{\om \in \Omega: X (\om) \in A\}) \quad \forall A \in \cE\,.
\end{equation*}
The image measure $\mu_X = \cL(X)$ on $(E,\cE)$ is called law or
distribution of $X$.

A random variable taking values in $E$ is called {\em simple} if it
can take only finitely many values, i.e., if it has the explicit form
(with $\chi_A$ the indicator function of $A \in \cF$)
\begin{equation*}
  X = \dis\sum\limits^N_{i=1} \,x_i \, \chi_{A_i}, \quad A_i \in \cF, \; x_i \in E, \; N < \infty\,.
\end{equation*}

\n We set, for simple random variables $X$ taking values in $E$ and
for any $B \in \cF$,
\begin{equation}\label{2.7}
  \dis\int_B X(\om) \, \IP(d\om) 
  = 
  \dis\int_B \,Xd \,\IP 
  := 
  \dis\sum\limits^N_{i=1} x_i \,\IP(A_i \cap B)\;.
\end{equation}
By density, for such $X(\cdot)$, and all $B \in \cF$,
\begin{equation}\label{2.8}
  \Big\| \dis\int_B X(\om) \, \IP(d\om)\Big\|_E 
  \le 
  \dis\int_B \|X(\om)\|_E \, \IP(d \om)\,.
\end{equation}

\medskip\n For any random variable $X: \Omega \rightarrow E$ which is
Bochner integrable, there exists a sequence $\{X_m\}_{m \in \IN}$ of
simple random variables such that, for all $\omega \in \Omega$,
$\|X(\om) - X_m(\om)\|_E \rt 0$ as $m \rightarrow \infty$. Therefore,
(\ref{2.7}) and (\ref{2.8}) extend in the usual fashion by continuity
to any $E$-valued random variable.  We denote the integral
\begin{equation*}
  \dis\int_\Omega X(\om) \,\IP(d\om) 
  = 
  \lim\limits_{m \rightarrow \infty} \dis\int_\Omega X_m (\om) \,\IP(d \om) \in E
\end{equation*}
by $\IE[X]$ (``expectation'' of $X$).
%
%
%
We shall require for $1 \le p \le \infty$ Bochner spaces of
$p$-summable random variables $X$ taking values in the Banach space
$E$.  By $L^1(\Omega, \cF, \IP; E)$ we denote the set of all
(equivalence classes of) integrable, $E$-valued random variables $X$,
equipped with the norm
\begin{equation*}
  \|X\|_{L^1(\Omega; E)} 
  = 
  \dis\int_\Omega  \|X(\om)\|_E \,\IP(d\om)
  =
  \IE (\|X\|_E) \;.
\end{equation*}
More generally, for $1\leq p < \infty$, we define $L^p(\Omega, \cF,
\IP; E)$ as the set of $p$-summable random variables taking values in
$E$ and equip it with norm
\begin{equation*}
  \|X\|_{L^p(\Omega; E)} : = (\IE(\|X\|^p_E))^{1/p}, \; 1 \le p < \infty\,.
\end{equation*}
For $p=\infty$, we denote by $L^\infty(\Omega, \cF, \IP; E)$ the set
of all $E$-valued random variables which are essentially bounded.
This set is a Banach space equipped with the norm
\begin{equation*}
  \|X\|_{L^\infty(\Omega; E)} 
  := 
 \essup_{\om \in \Omega} \|X(\om)\|_E
  \,.
\end{equation*}
If $T < \infty$ and $\Omega = [0,T]$, $\cF = \cB([0,T])$, we write
$L^p(0,T; E)$.  Note that for any separable Banach space $E$, and for
any $r \ge p \ge 1$,
\begin{equation*}
  L^r(0,T; E), \; C^0(0,T;E) \in \cB(L^p(0,T; E))\,.
\end{equation*}

In the following, we will be interested in random variables
$X:\Omega\rightarrow E_j$, $j=1,2$, mapping from some probability
space $(\Omega,\mathcal{F},P)$ into subsets of the Banach spaces
$E_j$, $j=1,2$, equipped with the Borel $\sigma$-algebra
$\mathcal{B}(E_j)$, where $E_1=L^1(\mathbb{R})\times C^1(I)\times
C^1(I)$, for a closed and bounded interval $I=[M_-,M_+]\subset (-\infty,\infty)$, $-\infty<M_-<M_+<\infty$,  and
$E_2=C([0,T];L^1(\mathbb{R}))$, $T>0$. 
On $C^1(I)$,  we choose the norm
\begin{equation*}
  \|f\|_{C^1(I)}=\sup_{x\in I}{|f(x)|} +\sup_{x\in I} {|f'(x)|}, \quad f\in C^1(I),
\end{equation*}
on $E_1$, 
we will use
\begin{equation*}
  \|\mathbf{g}\|_{E_1}=\|g_1\|_{L^1(\mathbb{R})}
  +\|g_2\|_{C^1(I)}+\|g_3\|_{C^1(I)},\quad 
  \mathbf{g}=(g_1,g_2,g_3)\in E_1, 
\end{equation*}
and on $E_2$,
\begin{equation*}
  \|h\|_{E_2}=\sup_{0\leq t\le T}\int_{\mathbb{R}}|h(t,x)|\, dx, \quad h\in E_2. 
\end{equation*}
Furthermore, we will need the following special case of the fact that
a continuous mapping is measurable:
\begin{lemma}\label{lem:continuity}
  Let $(\Omega,\mathcal{F})$ be a measurable space, $B_1, B_2$ be
  Banach spaces equipped with the Borel $\sigma$-algebras
  $\mathcal{B}(B_j)$, and let $j=1,2$, $X:\Omega\rightarrow B_1$ be a random
  variable and $\Psi:B_1\rightarrow B_2$ be a continuous mapping, that
  is for $x_1,x_2\in B_1$,
  \begin{equation*}
    \|\Psi(x_1)-\Psi(x_2)\|_{B_2}\le \lambda(\|x_1-x_2\|_{B_1}),
  \end{equation*}
where 
$\lambda:\R_{\geq 0} \rightarrow \R_{\geq 0}$ 
is a continuous function which satisfies 
$\lambda(0)=0$ and which increases monotonically in $\R_{\geq 0}$.

Then the mapping 
$\Omega\ni \omega\mapsto \Psi(X(\omega))\in B_2$ is 
$(B_2,\cB(B_2))$ - measurable, i.e., it is
a $B_2$-valued random variable.
\end{lemma}
\begin{proof}
  We have to show that for any $A_2\in \mathcal{B}(B_2)$,
  $F=(\Psi\circ X)^{-1}(A_2)\in \mathcal{F}$. Since $(\Psi\circ
  X)^{-1}(A_2)=X^{-1}(\Psi^{-1}(A_2))$ and by the assumption that $X$
  is a random variable $X^{-1}(A_1)\in \mathcal{F}$ for every $A_1\in
  \mathcal{B}(B_1)$, this amounts to showing that $\Psi^{-1}(A_2)\in
  \mathcal{B}(B_1)$ for any $A_2\in \mathcal{B}(B_2)$. Since the Borel
  $\sigma$-algebra is generated by the open sets and the inverse image
  of a mapping $f:\Omega_1\rightarrow \Omega_2$ has the two
  fundamental properties
  \begin{equation*}
    \bigcup_{i\in \mathcal{I}}f^{-1}(C_i)=f^{-1}(\bigcup_{i\in
      \mathcal{I}}C_i),\quad   \bigcap_{i\in
      \mathcal{I}}f^{-1}(C_i)=f^{-1}(\bigcap_{i\in \mathcal{I}}C_i), 
  \end{equation*}
  for a countable index set $\mathcal{I}$ and \cs{any countable collection 
   $\{C_i\}_{i\in\mathcal{I}}$ of sets $C_i\subset \Omega_2$},
  it is enough to verify this for an \cs{arbitrary open, nonempty
  set} $A_2\in \mathcal{B}(B_2)$. That $\Psi^{-1}(A_2)$ is
  an open set if $A_2$ is open, then follows by the continuity of
  $\Psi$.
\end{proof}
\section{Degenerate Convection Diffusion Equation with Random
          Diffusive Flux}
\label{sec:rand}
We develop a theory of random
entropy solutions for degenerate convection diffusion equation with a
class of random flux flunctions, proving in particular the existence
and uniqueness of a random entropy solution. 
\cs{
To this end, we first review classical results on degenerate convection diffusion
equation with deterministic data}.

\subsection{Deterministic Scalar Degenerate Convection Diffusion
  Equation}
\label{ssec:SCL}
We consider the Cauchy problem for degenerate convection diffusion
equation of the form
\begin{equation}
  \label{eq:main2}
  \begin{cases}
    u_t + \dv f(u) = \dv \left(a(u) \grad u \right), &(x,t)\in \Pt,\\
    u(0,x)=u_0(x), & x\in \R^d,
  \end{cases}
\end{equation}

\subsection{Entropy Solutions}
\label{ssec:EntrSol}
It is well-known that if $f$ is Lipschitz continuous and $a(u)\ge 0$,
then the deterministic Cauchy problem \eqref{eq:main2} admits, for
each $u_0 \in L^1(\R^d)\cap L^{\infty}(\R^d)$, a unique entropy
solution (see, e.g., \cite{GL,Smoller, Dafermos}).  Moreover, for
every $t > 0$, $u (\cdot, t)\in L^1(\IR^d) \cap L^{\infty}(\R^d)$ and
several properties of the (nonlinear) data-to-solution operator
\begin{equation*}
  S: (u_0,f,A) \longmapsto u(\cdot, t) = S(t) \,(u_0,f,A), \quad t > 0, 
\end{equation*}
will be crucial for our subsequent development.  
To state these properties of $\{S(t)\}_{t\ge 0}$, 
following \cite{ek2000} we introduce the set 
of admissible initial data
\begin{equation}
  \mathcal{A}(f,A):= \seq{ z \in L^1(\R^d) \cap BV(\R^d) \st \abs{f(z)
      -\grad A(z)}_{BV} < \infty }.  
  \label{eq:initialset}
\end{equation}

We collect next fundamental results regarding the 
entropy solution $u$ of \eqref{eq:main2} in 
the following theorem, for a proof see \cite{VolpertHudajev,chenperthame2003}, 
%
\begin{theorem}\label{theo3.1} ~
  Let $f$ and $A$ be locally Lipschitz continuous functions. Then
  \begin{itemize}
  \item[{\rm 1)}] For every $u_0 \in \mathcal{A}(f,A)$, the initial
    value problem \eqref{eq:main2} admits a unique BV entropy weak
    solution $u \in C\left( [0,T]; L^1_{\mathrm{loc}}(\R^d) \right)$.

  \item[{\rm 2)}] For every $t > 0$, the (nonlinear) data-to-solution
    map $S(t)$ given by
    \begin{equation*}
      u(\cdot,t) = S(t) \,(u_0,f,A)
    \end{equation*}
    satisfies
    \begin{itemize}
    \item[{\rm i)}] For fixed $f,A \in \mathrm{Lip}(\R)$,
      $S(t)(\cdot,f,A): L^1_{\mathrm{loc}}(\R^d) \rt L^1(\R^d)$ is a (contractive)
      Lipschitz map, i.e.,
      \begin{equation}\label{3.8}
        \norm{S(t) (u_0,f,A) - S(t) (v_0,f,A)}_{L^1(\R^d)} \le 
        \norm{u_0 - v_0}_{L^1(\R^d)}. 
      \end{equation}
    \item[{\rm ii)}] For every $u_0 \in \mathcal{A}(f,A)$, $f,A\in
      \mathrm{Lip}_{\mathrm{loc}}(\R)$
      \begin{align}
        \norm{S(t)(u_0,f,A)}_{L^\infty(\R^d)} & \le
        \|u_0\|_{L^\infty(\R^d)},
        \label{3.10}
        \\
        \norm{S(t)(u_0,f,A)}_{L^1(\R^d)}& \le \norm{u_0}_{L^1(\R^d)},
        \label{3.11}
        \\
        \norm{S(t)(u_0,f,A)}_{BV(\R^d)}& \le \norm{u_0}_{BV(\R^d)},
        \label{bv}
        \\
       \abs{f(u(\cdot,t))-\grad A(u(\cdot,t))}_{BV(\R^d)} &\le
        \abs{f(u_0)-\grad A(u_0)}_{BV(\R^d)} \label{eq:fluxbvu}.
      \end{align}

    \item[{\rm iii)}] Lipschitz continuity in time: For any $t_1,
      t_2>0$, $u_0\in \mathcal{A}(f,A)$,
      \begin{equation}\label{eq:lipt}
        \norm{S(t_1)(u_0,f,A)-S(t_2)(u_0,f,A)}_{L^1(\R^d)}\le
        \abs{f(u_0)-\grad A(u_0)}_{BV(\R^d)} \abs{t_1-t_2}. 
      \end{equation}
    \end{itemize}
  \end{itemize}
\end{theorem}
Point 1) of Theorem~\ref{theo3.1} is proved in \cite{VolpertHudajev}
or \cite[Thm 1.1]{chenperthame2003}, \eqref{3.8}, \eqref{3.11} also
follow from \cite[Thm 1.1]{chenperthame2003}, \eqref{3.10} was proved
in \cite[Thm 1.2]{chenperthame2003}, and \eqref{bv},
\eqref{eq:fluxbvu}, \eqref{eq:lipt} were proved in
\cite{VolpertHudajev}.
In our convergence analysis of MC-FD discretizations of degenerate
convection diffusion equation with random fluxes, we will need the
following result regarding continuous dependence of $S$ with respect
to $f$ and $A$ (\cite[Thm. 3]{cg1999})
\begin{theorem}\label{thm:fluxdependence}
  Assume $u_0$, $v_0 \in BV(\R^d)\cap L^1(\R^d) \cap
  L^{\infty}(\R^d)$, and $f(\cdot)$, $g(\cdot)$, $A(\cdot)$,
  $B(\cdot)$ $\in \Lip_{\mathrm{loc}}(\R)$ with $A', B'\ge 0$.

  Then the unique entropy solutions $u(t,\cdot)=S(t)(u_0,f,A)$ and
  $v(t,\cdot)=S(t)(v_0,g,B)$ of \eqref{eq:main2} with initial data
  $u_0$, $v_0$, convective flux functions $f$ and $g$ and with
  diffusive flux functions $A$ and $B$ satisfy the Kru\v{z}kov entropy
  conditions, and the \`a priori continuity estimate
  \begin{align}\label{eq:continentrsol}
    \norm{u(\cdot,t) - v(\cdot,t)}_{L^1(\R^d)} & \le\norm{u_0 - v_0}_{L^1(\R^d)} 
    \\
    & + C \left( t \norm{f' - g'}_{L^{\infty}(M_-, M_+)} + 4 \sqrt{t}
      \norm{\sqrt{A'} - \sqrt{B'}}_{L^{\infty}(M_-, M_+)}\right)\notag,
  \end{align}
where $M_- \le u_0 \le M_+$ and $C = \abs{u_0}_{BV(\R^d)}< \infty$. The
above estimate holds for every $0\le t \le T$.
\end{theorem}
\begin{remark}\label{rem:cd}
Using that for nonnegative numbers $a,b\ge 0$, $a\neq 0$,
\begin{equation*}
  |\sqrt{a}-\sqrt{b}|=\sqrt{|a-b|}\frac{\sqrt{|a-b|}}{\sqrt{a}+\sqrt{b}}\le
  \sqrt{|a-b|}, 
\end{equation*}
it follows from \eqref{eq:continentrsol} that under the assumptions
of Theorem~\ref{thm:fluxdependence},
\begin{align*}
\norm{u(\cdot,t) - v(\cdot,t)}_{L^1(\R^d)} & 
\le
\norm{ u_0 - v_0}_{L^1(\R^d)} 
\\
 & + C \left( t \norm{ f' - g'}_{L^{\infty}(M_-,M_+)} + 4 \sqrt{t}
  \sqrt{\norm{{A'} - {B'}}_{L^{\infty}(M_-,M_+)}}\,\right)\notag,
\end{align*}
hence the mapping 
$S(t):L^1(\R^d)\times W^{1,\infty}([M_-,M_+])\times W^{1,\infty}([M_-,M_+])\rightarrow L^1(\R^d)$, 
$(u_0,f,A)\mapsto u(t,\cdot)$ is continuous as a mapping between Banach spaces if
restricted to initial data $u_0$ in 
$U_1:=\seq{ u_0\in L^1(\R^d):\, M_-\le u_0(x)\le M_+, 
\text{a.e.~$x\in \R$}}\subset \cs{L^1(\IR^d)}$ and 
$A$ satisfying $A'\ge 0$. 
Moreover, since for $f,g,A,B\in C^1([M_-,M_+])$
with the above properties and bounded derivatives
\cs{it holds}
\begin{align}\label{eq:continentrsol3}
  \norm{u(\cdot,t) - v(\cdot,t)}_{L^1(\R^d)} & \le
  \norm{u_0 - v_0}_{L^1(\R^d)} 
  \\
  & + C \left( t \sup_{z\in [M_-,M_+]}| f'(z) - g'(z)| + 4 \sqrt{t}
    \sqrt{\sup_{z\in [M_-,M_+]}| {A'(z)} - {B'(z)}|}\right)\notag,
\end{align}
\cs{it follows} that $S(t)(\cdot,\cdot,\cdot)$ is a continuous mapping
between the separable Banach spaces $E_1:=L^1(\R^d)\times
C^{1}([M_-,M_+])\times C^{1}([M_-,M_+])$ and 
$L^1(\R^d)$ if restricted to
initial data in the set $U_1$.
\end{remark}

\subsection{Random Entropy Solutions}
\label{ssec:RndInit}
We are interested in the case where the initial data $u_0$, the
convective flux function $f$ and the diffusive flux function $A$ in
\eqref{eq:main1} are uncertain.  Existence and uniqueness for random
initial data $u_0$ and random flux $f$ for $A\equiv 0$ was proved in
\cite{MishSch10a,Mishr478}.  Based on Theorem~\ref{theo3.1}, 
we will now formulate \eqref{eq:main1} for random initial data
$u_0(\omega;\cdot)$, random convective flux $f(\omega;\cdot)$ and
random diffusive flux $A(\omega;\cdot)$. To this end, we denote
$(\Omega, \cF, \IP)$ a probability space and consider random variables
\begin{equation}\label{eq:rv}
  \begin{cases}
   X:=(u_0,f,A):(\Omega,\mathcal{F})\longrightarrow (E_1,\mathcal{B}(E_1)),
   \\
   \omega\longmapsto X(\omega):=(u_0(\omega;\cdot),f(\omega; \cdot),A(\omega; \cdot)),
  \end{cases}
\end{equation} 
where $E_1:=L^1(\R)\times C^1(I)\times C^1(I)$ and $I:=[M_-,M_+]\subset \R$, $-\infty<M_-<M_+<\infty$. In order to establish the appropriate framework for the
random degenerate convection diffusion equation \eqref{eq:RSCL}, we
will restrict ourselves to random data
$X=(u_0,f,A)\in E_1$ which satisfy 
$\IP$-a.s. the following assumptions:
\begin{align}
  &-\infty < M_- \le u_0(\omega;x)\le M_+ < \infty,\ \text{a.e.}\ x\in \R^d,
  \label{eq:assu}\\
  &\abs{u_0(\omega;\cdot)}_{BV(\R^d)}\le C_{\mathrm{TV}}<\infty,
  \label{eq:assu2}\\
  &\norm{f'(\omega;\cdot)}_{C^0([M_-,M_+])}
  \le C_f<\infty,\label{eq:assf}\\
  &A'(\omega;\cdot)\ge 0, \label{eq:assA}
  \\
  &\norm{A'(\omega;\cdot)}_{C^0([M_-,M_+])}\le C_A<\infty,
  \label{eq:assA2}
  \\
  &\abs{f(\omega;u_0(\omega;\cdot))-\grad
    A(\omega;u_0(\omega;\cdot))}_{BV(\R^d)}\le
  C_{A,f}<\infty.\label{eq:assAf}
\end{align}

Since $L^1(\R^d)$ and $C^1(I)$ are separable, 
\eqref{eq:rv} is well defined.
\cs{
Moreover, by Lemma  \cite[Lemma 1.5, p.19]{DaPrZab91} 
each of the expressions on the left hand sides
of \eqref{eq:assu} - \eqref{eq:assAf} is a random variable
and we may 
impose for $k \in \IN$ the \emph{$k$-th moment condition}:
}
\begin{equation}\label{3.16}
  \norm{u_0}_{L^k(\Omega; L^1(\R^d))} < \infty,
\end{equation}
where the Bochner spaces with respect to the probability measure are
defined in Section~\ref{sec:prob}. Then we are interested in random
solutions of the \emph{random degenerate convection diffusion
  equation}
\begin{equation}\label{eq:RSCL}
  \begin{cases}
    u_t(\omega;x,t) + \dv (f(\omega;u(\omega;x,t))) = \Delta
    A(\omega;u(\omega; x,t)),
    \  t>0, \, x\in \R^d,\\
    u(\omega;x,0) = u_0(\omega;x), \ x\in \R^d.
  \end{cases}
\end{equation}
\begin{definition}
  \label{def:res}
  A random field $u:\Omega\ni\omega \rightarrow u(\omega;x,t)$, i.e.,
  a measurable mapping from $(\Omega,\cF)$ to $C([0,T];L^1(\R^d))$, is
  called a \emph{random entropy solution} of \eqref{eq:main2} with
  random initial data $u_0$, flux function $f$ and diffusive flux $A$
  satisfying \eqref{eq:rv} and \eqref{eq:assu} -- \eqref{3.16} for
  some $k\ge 2$, if it satisfies:
  \begin{itemize}
  \item [(i.)] Weak solution: 
    for $\IP$-a.e.~$\omega \in \Omega$,
    $u(\omega;\cdot,\cdot)$ satisfies
    \begin{multline*}
      \int\limits_{0}^{\infty} \int\limits_{\R^d} \Bigl(u(\omega;x,t)
      \test_t + \left(f(\omega;u(\omega;x,t)) - \grad A(\omega;
        u(\omega; x,t)) \right)\cdot \grad \test \Bigr)\, dx dt
      \\
      + \int\limits_{\R^d} u_0(x,\omega) \test(x,0) \,dx = 0,
    \end{multline*}
    for all test functions $\test \in C^1_0(\R^d\times [0,\infty))$.

  \item[(ii.)] Entropy condition: For any pair consisting of a
    (deterministic) entropy $\eta$ and (stochastic) entropy flux
    $q(\omega;\cdot)$ and $r(\omega;\cdot)$ i.e., $\eta,q$ and $r$ are
    functions such that $\eta$ is convex and such that
    $q^{\prime}(\omega;\cdot) =
    \eta^{\prime}f^{\prime}(\omega;\cdot)$, $r^{\prime}(\omega;\cdot)
    = \eta^{\prime}A^{\prime}(\omega;\cdot)$and for $\IP$-a.s.~$\omega
    \in \Omega$, $u$ satisfies the following integral identity:
    \begin{multline*}
      \int\limits_{0}^{\infty} \int\limits_{\R^d}
      \Bigl(\eta(u(\omega;x,t)) \test_t + \grad q(\omega;
      u(\omega;x,t))\cdot \grad\test + r(\omega; u(\omega;x,t)) \Delta
      \test \Bigr) \,dx dt
      \\
      +\int\limits_{\R^d} \eta(u_0(\omega;x))\test(x,0)\,dx \ge 0,
    \end{multline*}
    for all test functions $0 \le \test \in C^1_0(\R^d\times
    [0,\infty))$.
  \end{itemize}
\end{definition}
We state the following theorem regarding the random entropy solution
of \eqref{eq:RSCL}:
\begin{theorem}\label{theo3.2}
  Consider the degenerate convection diffusion equation
  \eqref{eq:main2} with random initial data $u_0$, flux function $f$
  and random diffusion operator $A$, as in \eqref{eq:rv}, and
  satisfying \eqref{eq:assu} -- \eqref{eq:assAf} and the $k$-th moment
  condition \eqref{3.16} for some integer $k\ge 2$.  Then there exists
  a unique random entropy solution $u : \Omega \ni\omega \rt C([0, T];
  L^1(\R))$ which is ``pathwise'', i.e., for $\IP-\text{a.s.}\
  \omega\in \Omega$, described in terms of a nonlinear mapping $S(t)$,
  depending only on the random flux and diffusion,
  \begin{equation*}
    u(\omega;\cdot, t) = S(t)(u_0(\omega;\cdot),f(\omega;\cdot),A(\omega;\cdot)),
    \quad t > 0,\; \IP-\text{a.e.}\ \om\in\Omega
  \end{equation*}
  such that for every $k \ge m \ge 1$ and for every $0 \le t \le T<
  \infty$
  \begin{align}
    \norm{u}_{L^k\left(\Omega; C([0,T];L^1(\R^d))\right)} &\le
    \norm{u_0}_{L^k\left(\Omega; L^1(\R^d)\right)},\label{3.18}
    \\
    \norm{S(t) (u_0,f,A)(\omega)}_{(L^1 \cap L^\infty)(\R^d)} & \le
    \norm{u_0(\omega;\cdot)}_{(L^1 \cap L^\infty)(\R^d)}
    \label{3.19}
  \end{align}
  and such that we have $\IP$-a.s.
  \begin{align}
    \label{eq:uno}
    \abs{S(t) (u_0,f,A)(\omega)}_{BV(\R^d)} &\le
    \abs{u_0(\omega;\cdot)}_{BV(\R^d)},\\
    \label{eq:tres}
    \abs{f(\omega;u(\omega;\cdot,t))-\grad
      A(\omega;u(\omega;\cdot,t))}_{BV(\R^d)} &\le
    \abs{f(\omega;u_0(\omega;\cdot))-\grad
      A(\omega;u_0(\omega;\cdot))}_{BV(\R^d)},
    \\
    \label{eq:quatro}
    \norm{u(\omega;\cdot,t_1)-u(\omega;\cdot,t_2)}_{L^1(\R^d)}&\le
    \abs{f(\omega;u_0(\omega;\cdot))-\grad
      A(\omega;u_0(\omega;\cdot))}_{BV(\R^d)} \abs{t_1-t_2}.
  \end{align}
  and, with $\overline{M}:=\max\{¦M_-¦,¦M_+¦\}$ for $M_-,M_+$ as in \eqref{eq:assu},
  \begin{equation}\label{eq:utBound}
    \sup_{0\le t \le T} \norm{u(\omega;\cdot,t)}_{L^\infty(\R^d)} 
    \le \overline{M} \quad \text{$\IP$-a.s.~$\omega\in \Omega$ 
      .}
  \end{equation}
\end{theorem}
\begin{proof}
  For $\omega\in \Omega$, we define, motivated by Theorem
  \ref{theo3.1}, for $\mathbb{P}$-a.e.~$\omega\in \Omega$ a random
  function $u(\omega;t, x)$ by
  \begin{equation}\label{eq:res}
    u(\omega;\cdot) = S(t)(u_0,f,A)(\omega).
  \end{equation}
  By the properties of the solution mapping $(S(t))_{t\ge 0}$, see
  Theorem~\ref{theo3.1}, the random field defined in \eqref{eq:res} is
  well defined; for $\mathbb{P}$-a.e.  $\omega\in \Omega$,
  $u(\omega;\cdot)$ is a weak entropy solution of the degenerate
  diffusion equation \eqref{eq:main2}. Moreover, we obtain from
  Theorem~\ref{theo3.1} that $\mathbb{P}$-a.s.~all bounds
  \eqref{3.19}--\eqref{eq:quatro} hold, with assumption
  \eqref{eq:assu} also \eqref{eq:utBound}. The measurability of the
  mapping $\Omega\ni\omega\mapsto u(\omega;\cdot,t)\in L^1(\R)$, $0\le
  t\le T$ follows from Lemma \ref{lem:continuity},
  \eqref{eq:continentrsol3} and the assumption that the mapping
  $\Omega\ni\omega\mapsto (u_0,f,A)(\omega)\in E_1$ is a random
  variable. Finally, \eqref{3.18} follows from \eqref{3.16} together
  with \eqref{3.11} in Theorem \ref{theo3.1}.
\end{proof}
Theorem~\ref{theo3.2} generalizes the existence of random entropy
solutions for random initial data from \cite{MishSch10a} and random
convective flux function \cite{Mishr478}.  It ensures the existence of
a unique random entropy solution $u(\omega;x,t)$ with finite $k$-th
moments provided that $u_0\in L^k(\Omega, \cF, \IP; L^1(\R^d))$ for some
$k\ge 2$.
 
\begin{remark} \label{remk:Cauchy_Periodic}
All existence and continuous dependence results stated so far
are formulated for the deterministic Cauchy problem \eqref{eq:main2}.
By the `usual arguments', verbatim the same results will also hold
for solutions defined in a bounded, axiparallel domain $D\subset \R^d$,
provided that periodic boundary conditions in each coordinate
are enforced on the weak solutions. 
Weak solutions for these periodic problems 
cannot coincide with weak solutions of the Cauchy problem
\eqref{eq:main2} since the $D$-periodic extension of these solutions
\cs{belongs to $L^1_{loc}(\R^d)$, but} does not belong to $L^1(\R^d)$. 
\end{remark}

\section{Numerical approximation of random degenerate convection
  diffusion equation}\label{sec:schemes}
We wish to compute various quantities of interest, such as the
expectation and higher order moments, of the solution $u$ to the
random degenerate diffusion equation \eqref{eq:RSCL}. We choose to
split the approximation into two steps: On one hand, we need to
approximate in the stochastic domain $\omega\in \Omega$ and on the
other hand, since in general exact solutions to \eqref{eq:main1} are
not available, we need an approximation in the physical domain
$(x,t)\in \Pi_T$. In this paper, we will consider a
Multilevel Monte Carlo Finite Difference Method (MLMC-FDM), that is, a
combination of the \cs{multilevel} Monte Carlo method with a
deterministic finite difference \cs{discretization}. 
We will briefly review the two
methods and mention some relevant results in the following sections.
\subsection{Monte Carlo method}\label{subsec:MCm}
We view the Monte Carlo method as a ``discretization'' of the random
degenerate diffusion equation data $u_0(\omega;\cdot)$,
$f(\omega;\cdot)$, $A(\omega; \cdot)$ with respect to $\omega\in
\Omega$. We assume that
$(u_0(\omega;\cdot),f(\omega;\cdot),A(\omega;\cdot))\in E_1$
satisfying in addition \eqref{eq:assu}--\eqref{eq:assAf}. We also
assume (\ref{3.16}), i.e., the existence of $k$-th moments of $u_0$
for some $k \in \IN$, to be specified later.  We shall be interested
in the statistical estimation of the first and higher moments of $u$,
i.e., $\cM^k(u) \in (L^1(\R^d))^{(k)}$.  For $k = 1$, $\cM^1(u) = \IE
[u]$.  The {\em Monte Carlo (MC) approximation of $\IE[u]$} is defined
as follows: Given $M$ independent, identically distributed samples
$(\wh{u}_0^i,\wh{f}^i,\wh{A}^i)$, $i = 1,\dots, M$, of initial data,
flux function and diffusion, the MC estimate of
$\IE[u(\cdot,t;\cdot)]$ at time $t$ is given by
\begin{equation}\label{4.1}
  E_M [u(\cdot, t)] : = \dis\frac{1}{M} \; \dis\sum\limits^M_{i=1}
  \,\wh{u}^i(\cdot,t) 
\end{equation}
where $\wh{u}^i(\cdot,t)$ denote the $M$ unique entropy solutions of
the $M$ Cauchy problems \eqref{eq:main1} with initial data
$\wh{u}^i_0$, flux function $\wh{f}^i$ and diffusion operator
$\wh{A}^i$.  Since
\begin{equation*}
  \wh{u}^i(\cdot, t) = S(t) \, (\wh{u}_0^i,\wh{f}^i,\wh{A}^i),
\end{equation*}
we have for every $M$ and for every $0 < t < \infty$, by (\ref{3.11}),
\begin{equation*}
  \begin{split}
    \|E_M [u(\cdot,t;\omega)]\|_{L^1(\R^d)} & = \Big\| \dis\frac{1}{M}
    \; \dis\sum\limits^M_{i=1} \,
    S(t)(\wh{u}_0^i,\wh{f}^i,\wh{A}^i)(\omega)\Big\|_{L^1(\R^d)}\\
    &\le \dis\frac{1}{M} \; \dis\sum\limits^M_{i=1} \,\left\|
      S(t)(\wh{u}_0^i,\wh{f}^i,\wh{A}^i)(\omega)\right\|_{L^1(\R^d)}
    \\[1ex]
    & \le \dis\frac{1}{M} \; \dis\sum\limits^M_{i=1} \, \|
    \wh{u}_0^i(\cdot;\omega)\|_{L^1(\R^d)}\,.
  \end{split}
\end{equation*}
Using the i.i.d.~property of the samples
$\seq{(\wh{u}^i_0,\wh{f}^i,\wh{A}^i)}^M_{i=1}$ and therefore of
$\{\wh{u}^i_0\}^M_{i=1}$, and the linearity of the expectation
$\IE[\cdot]$, we obtain the bound
\begin{equation*}
  \IE\left[\norm{E_M[u(\cdot,t)]}_{L^1(\R^d)}\right] 
  \le 
  \IE\left[\norm{u_0}_{L^1(\R^d)}\right] = 
  \norm{u_0}_{L^1(\Omega; L^1(\R^d))} < \infty.
\end{equation*}
As the sample size $M \rt \infty$, 
the sample averages \eqref{4.1} converge and the
convergence result from \cite{MishSch10a,Mishr478} holds as well:

\begin{theorem}\label{theo4.1}
  Assume that in \eqref{eq:RSCL} the random variable
  $(u_0,f,A)(\omega)$ as in \eqref{eq:rv} satisfies \eqref{eq:assu}
  and $A'(\omega;\cdot)\ge 0$, a.s.~$\omega\in\Omega$ and
  \begin{equation*}
    u_0 \in L^2(\Omega; L^1(\R^d)).
  \end{equation*}
  Then the MC estimates $E_M[u(\cdot,t)]$ in \eqref{4.1} converge as
  $M \rt \infty$, to $\cM^1(u(\cdot, t)) = \IE[u(\cdot,t)]$ and, for
  any $M \in \IN$, $0 < t < \infty$, and we have the bound
  \begin{equation}\label{4.6}
    \norm{ \IE[u(\cdot,t)] - E_M[u(\cdot,t)]}_{L^2(\Omega;L^1(\R^d))} 
    \le 
    2M^{-\frac{1}{2}} \norm{u_0}_{L^2(\Omega; L^1(\R^d))}.
  \end{equation}
\end{theorem}
The proof of this result proceeds completely analogous to the 
proof of \cite[Thm. 4.1]{MishSch10a}, using the 
measurability and square integrability \eqref{3.18} 
(with $k=2$) of Theorem~\ref{theo3.2}.


So far, we addressed the MC estimation of the mean field or first
moment. A similar result holds for the MC sample averages of the
$k$-th moment $\cM^k u : = \IE[(u)^{(k)}] \in (L^1(\R^d))^{(k)}$.
\begin{theorem}\label{theo3.4}
  Consider the random degenerate advection diffusion equation
  \eqref{eq:RSCL} with random data $(u_0,f,A) : \Omega \rt E_1$ as in
  \eqref{eq:rv} and satisfying \eqref{eq:assu} and
  $A'(\omega;\cdot)\ge 0$, a.s.  Assume furthermore that for some $k
  \in \IN$ holds $u_0 \in L^{2k} (\Omega; L^1(\R^d))$.  Then, as
  $M\rightarrow \infty$, the MC sample averages
  \begin{equation*}
    E_M[(u(\cdot,t))^{(k)}]
    := 
    \dis\frac{1}{M} \; \dis\sum\limits^M_{i=1} \;(\wh{u}^i(\cdot,t))^{(k)}
  \end{equation*}
  with the $M$ i.i.d.~samples $\wh{u}^i(\cdot,t)$, $i=1,2,...$,
  converge to the $k$-th moment (or spatial $k$-point correlation
  function) $(\cM^k u)(t)$. Moreover, we have the error bound
  \begin{equation*}
    \norm{(\cM^k u)(t) - E_M[(u(\cdot,t;\omega))^{(k)}]}_{L^2(\Omega;L^1(\R^{kd}))}
    \le 
    2 M^{-1/2}
    \norm{ u_0 }^k_{L^{2k}(\Omega;L^1(\R^d))} 
    \;.
  \end{equation*}
\end{theorem}
The proof of this theorem is \cs{omitted since it is} 
identical to the proof of Theorem~4.2. in \cite{MishSch10a}.

\subsection{Finite Difference Methods for degenerate convection
  diffusion equations} 
So far, we considered the MCM under the assumption that the entropy
solutions $\wh{u}^i(x,t; \omega) = S(t)
\,(\wh{u}_0^i,\wh{f}^i,\wh{A}^i)(\omega)$ for the Cauchy problem
\eqref{eq:main1} with the data samples
$(\hat{u}_0^i,\wh{f}^i,\wh{A}^i)$ are available exactly.  In practice,
however, we must use numerical approximations of
$S(t)(\hat{u}_0^i,\wh{f}^i,\wh{A}^i)$.

The presentation will, from now on,
be restricted to the one-dimensional case, 
i.e., we consider
\begin{equation}
  \label{eq:1dmain}
  \begin{cases}
    u_t + f(u)_x = A(u)_{xx}, \quad t>0, \ x\in \R,
    \\
    u(x,0)=u_0(x).
  \end{cases}
\end{equation}
We shall examine the class of fully discrete monotone difference
schemes for which Karlsen, Risebro and Storr\o{}sten obtained a
convergence in $L^1$ rate of $\Dx^{1/3}$, where $\Dx$ is the
discretization parameter, in \cite{KarlsenRisebroStorrosten1}. These
schemes are easily generalized to several space dimensions, but
rigorous results regarding convergence rates are much worse. To date,
the best convergence rate in $L^1(\R^d)$ for a fully discrete, implicit
in time scheme is $\Dx^{1/(11+d)}$, see
\cite{KarlsenRisebroStorrosten2}.

For $\Delta x,\Dt>0$, we discretize the physical domain $\Pi_T$ by a
grid with grid cells
\begin{equation*}
  I^n_j=[x_{j-1/2},x_{j+1/2})\times (t_{n-1},t_n], \quad n\ge  0,\, j\in\mathbb{Z},
\end{equation*}
where $x_{j\pm 1/2}=(j\pm 1/2)\Dx$, $j\in\mathbb{Z}$, and $t_n=n\Dt$, $n\in
\mathbb{N}$. We define cell averages of the initial data via
\begin{equation}\label{eq:ini}
  u_j^0= \frac{1}{\Dx}\int_{I_j^0}u_0(x)\,dx,\quad j\in\mathbb{Z}.
\end{equation}
Then we consider the following implicit scheme
\begin{equation}\label{eq:IMP}
  \Dm^t u^n_j+ \Dm F\left(u_j^{n},u_{j+1}^n\right)=\Dm\Dp
  A(u_j^n),\quad  n\ge  1,\, j\in\mathbb{Z}, 
\end{equation}
and the explicit scheme,
\begin{equation}\label{eq:EXP}
  \Dp^t u^n_j+ \Dm F\left(u_j^{n},u_{j+1}^n\right)=\Dm\Dp
  A(u_j^n),\quad n\ge  0,\, j\in\mathbb{Z}, 
\end{equation}
where we have denoted for a quantity
$\{\sigma_j^n\}_{j\in\mathbb{Z},n\in\mathbb{N}}$,
\begin{equation*}
  D^t_\pm \sigma^n_j=\pm\frac{1}{\Dt} (\sigma^{n\pm
    1}_j-\sigma_j^n),\quad D_{\pm}
  \sigma_j^n=\pm\frac{1}{\Dx}(\sigma_{j\pm 1}^n-\sigma_j^n). 
\end{equation*}
We then define the piecewise constant approximation to
\eqref{eq:1dmain} by
\begin{equation}\label{eq:pwc}
  u_\Delta(x,t)=u^n_j,\quad (x,t)\in I^n_j,
\end{equation}
where $u^n_j$ is defined by either \eqref{eq:IMP} or
\eqref{eq:EXP}. The numerical flux $F\in C^1(\R^2)$ is chosen such
that it is consistent with $f$, that is, $F(u,u)=f(u)$ for all
$u\in\R$, and monotone, i.e.
\begin{equation*}
  \frac{\partial}{\partial u} F(u,v)\ge  0\quad \text{and}\quad
  \frac{\partial}{\partial v} F(u,v)\le 0. 
\end{equation*}
In order to obtain convergence rates, it is furthermore necessary to
choose $F$ Lipschitz continuous and such that it can be written
\begin{equation}\label{eq:zerlegung}
  F(u,v)=F_1(u)+F_2(v), \quad F'_1(u)+F'_2(u)=f'(u),
\end{equation}
see \cite{KarlsenRisebroStorrosten1}. 
Examples of monotone numerical fluxes satisfying \eqref{eq:zerlegung}
are the Engquist-Osher flux as well as the Lax-Friedrichs and the
upwind flux. In order to show convergence of the explicit scheme, the
following CFL-condition is needed,
\begin{equation}\label{eq:cflexp2}
  \Dt\le C \Dx^2,
\end{equation}
\cite{ek2000} and in order to show a convergence rate, one even needs
%
\begin{equation}\label{eq:cflexp}
  \Dt\le C \Dx^{8/3},
\end{equation}
see \cite{KarlsenRisebroStorrosten1}.
Whether this restrictive CFL-condition \cs{is sharp in order} 
to prove a convergence rate is not known. 
\cs{Naturally,} no CFL-condition is needed \cs{to ensure stability of} 
the implicit scheme, \cite{ekr2000}. 
In order to obtain \`a priori estimates for the
explicit scheme, the numerical flux function $F$ and the diffusion
operator $A$ have to satisfy the following condition
\begin{equation}\label{eq:condexp}
  \frac{\Dt}{\Dx} (F'_1(z)-F'_2(z))+2\frac{\Dt}{\Dx^2} A'(w)\le 1,\quad
  \text{for all $z$ and $w$,} 
\end{equation}
see \cite{ek2000}.
Then we have the following stability and convergence results for the
schemes \eqref{eq:IMP} and \eqref{eq:EXP}, 
\cite{ek2000,ek1998,KarlsenRisebroStorrosten1}
%
%
\begin{theorem}\label{thm:diffschemes}
  Let $u_0\in BV(\R)\cap L^{1}(\R)$, $f$, $A$ locally $C^1$, $A'\ge 0$
  and $u_0\in \mathcal{B}(f,A)$, where $\mathcal{B}(f,A)$ is defined
  in \eqref{eq:initialset}. Let $F$ be a monotone numerical flux
  function consistent with $f$, satisfying
  \eqref{eq:zerlegung}. Denote by $u_\Delta(x,t)$ the piecewise
  constant function defined in \eqref{eq:pwc}, where $u_j^n$ are
  computed by either the explicit scheme \eqref{eq:EXP} or the
  implicit scheme \eqref{eq:IMP}. Assume for the explicit scheme in
  addition that $\Dt$ satisfies \eqref{eq:cflexp2} and that
  \eqref{eq:condexp} holds. Then we have
  \begin{itemize}
  \item[{\rm i)}] The approximations $u_\Delta$ converge, as the
    discretization parameters $(\Delta x,\Dt)\rightarrow 0$ to the
    unique entropy solution of \eqref{eq:1dmain}. Moreover they
    satisfy
    \begin{align*}
      \norm{u_\Delta(\cdot,t)}_{L^1(\R)}&\le \norm{u_0}_{L^1(\R)},\\
      \norm{u_\Delta(\cdot,t)}_{L^\infty(\R)}&\le \norm{u_0}_{L^\infty(\R)},\\
      \abs{u_\Delta(\cdot,t)}_{BV(\R)}&\le \abs{u_0}_{BV(\R)},\\
      \sup_j \abs{F(u^n_j,u^n_{j+1})-\Dp A(u^{n}_j)}&\le
      \sup_j \abs{F(u^0_j,u^0_{j+1})-\Dp A(u^{0}_j)},\\
      \sum_j \abs{\Dm F(u^n_j,u^n_{j+1})-\Dm\Dp A(u^n_j)}&\le \sum_j
      \abs{\Dm F(u^0_j,u^0_{j+1})-\Dm\Dp A(u^0_j)}.
    \end{align*}
    Furthermore, $u_\Delta$ is $L^1(\R)$-Lipschitz continuous in time,
    viz., for any $t_n$, $t_m>0$,
    \begin{equation*}
      \norm{u_\Delta(\cdot,t_n)-u_\Delta(\cdot,t_m)}_{L^1(\R)}\le
      \abs{f(u_0)-A(u_0)}_{BV(\R)} \abs{t_n-t_m}.
    \end{equation*}
  \item[{\rm ii)}] If for the explicit scheme in addition
    \eqref{eq:cflexp} holds, the approximations $u_\Delta$ converge at
    the rate $1/3$ to the entropy solution $u$ of \eqref{eq:1dmain}:
    \begin{equation*}
      \norm{u_\Delta(\cdot,t_n)-u(\cdot,t_n)}_{L^1(\R)}\le 
      \norm{u_\Delta(\cdot,0)-u_0}_{L^1(\R)}+C_T\Dx^{1/3},
    \end{equation*}
    where the constant $C_T$ takes the form
    \begin{equation*}
      C(1+T)\left((1+\norm{f}_{\mathrm{Lip}})\abs{u_0}_{BV(\R)}
        +\norm{A(u_0)_x}_{L^1(\R)}+\abs{f(u_0)-A(u_0)_x}_{BV(\R)}\right), 
    \end{equation*}
    with $C$ independent of $u_0$, $f$ and $A$.
  \end{itemize}
\end{theorem}
Point i) was proved in \cite[Thm. 3.9, Cor. 3.10]{ek2000} for the
explicit scheme and \cite[Thm. 3.9, Lem. 3.3, 3.4, 3.5]{ek1998} for
the implicit scheme , ii) in \cite{KarlsenRisebroStorrosten1}.

For the purpose of analyzing the efficiency of the MC- and
MLMC-method, it is important to have an estimate on the computational
work used to compute one approximation of the solution by the
deterministic FD-schemes and how it increases with respect to mesh
refinement. By \emph{(computational) work} or \emph{cost} of an
algorithm, we mean the number of floating point operations performed
during the execution of the algorithm. We assume that this is
proportional to the run time of the algorithm. In the actual
computations we deal with bounded domains, so that the number of grid
cells in one dimension scales as $1/\Dx$.
\subsubsection{Work estimate explicit scheme
  \eqref{eq:EXP}}\label{ssec:workexp}
In case of the explicit scheme, the number of operations in one time
step scales linearly with the number of grid cells which in turn
scales as $\Dx^{-1}$ (we assume the computational domain is
bounded). 
Hence the work \cs{can be bounded as}
$W_\Delta^{\mathrm{exp}} \cs{\leq} C \Dt^{-1}\Dx^{-1}$. 
Taking the CFL-condition \eqref{eq:cflexp} into account, 
we obtain \cs{the (likely pessismistic) work bound}
\begin{equation*}
  W_{\Delta}^{\mathrm{ex}}= {\mathcal{O}(\Dx^{-11/3})}
\;.
\end{equation*}
\subsubsection{Work estimate implicit scheme \eqref{eq:IMP}}
\label{ssec:workimp}
In the implicit scheme we have to solve the nonlinear equation
\eqref{eq:IMP} for
$\underline{u}^{n+1}:=(\dots,u^{n+1}_{j-1},u^{n+1}_j,u^{n+1}_{j+1},\dots)$
in each timestep. Since solving this equation exactly is either
impossible or computationally very expensive, we prefer to solve it
only approximately by an iterative method. We consider here the case
that this method is the Newton iteration, which we iterate until the
residual is of order $\Dx \Dt$ (this is possible since the mapping
$\underline{u}^n\rightarrow
\underline{u}^{n+1}=:\Psi(\underline{u}^n)$ defined by \eqref{eq:IMP}
is a contraction \cs{for sufficiently small $\Dt$ and CFL constant.}
\cs{
In general the Lipschitz constant should scale as $1/\Dx$, 
so a small value of $\Dt$ alone is not sufficient for the
contraction property to hold.
For details, we refer to \cite{ek1998}}. 
\cs{The additional error introduced by finite termination of
the iterative nonlinear system solver
will not increase the overall error:}
denoting by $\underline{u}^{n,(0)}$
the approximation at time $t=t_n$ obtained by solving \eqref{eq:IMP}
exactly in each time step, $\underline{u}^{n,(j)}$ the approximation
obtained by solving \eqref{eq:IMP} approximately via Newton iteration
in the first $j$ timesteps and afterwards exactly (so that
$\underline{u}^{n,(n)}:=u_\Delta(\cdot,t_n)$ is the approximation
obtained by using Newton's method in each timestep), we have
\begin{align*}
  \norm{u_\Delta(\cdot,t_n)-u(\cdot,t_n)}_{L^1(\R)}
  &=\norm{\underline{u}^{n,(n)}-u(\cdot,t_n)}_{L^1(\R)}
  \\
  &=\biggl\|\sum_{m=0}^{n-1}
  (\underline{u}^{n,(m+1)}-\underline{u}^{n,(m)})
  +\underline{u}^{n,(0)}-u(\cdot,t_n)\biggr\|_{L^1(\R)}
  \\
  &\le \sum_{m=0}^{n-1}
  \norm{\underline{u}^{n,(m+1)}-\underline{u}^{n,(m)}}_{L^{1}(\R)}
  +\norm{\underline{u}^{n,(0)}-u(\cdot,t_n)}_{L^1(\R)}
  \\
  &\le \sum_{m=0}^{n-1}
  \norm{\underline{u}^{m+1,(m+1)}-\underline{u}^{m+1,(m)}}_{L^{1}(\R)}
  +C_T \Dx^{1/3}\\
  &\le n\Dx\Dt+ C_T \Dx^{1/3}\\
  &=t_n \Dx+ C_T \Dx^{1/3}\le \widetilde{C}_T \Dx^{1/3},
\end{align*}
where we have used the $L^1$-contraction property of the scheme for
the third last inequality. If the starting value for the Newton
iteration is chosen such that it is in a sufficiently small neighborhood of
the fixpoint,  the convergence order
of the Newton method is \cs{locally} quadratic. 
\cs{In} order to achieve an
error of less than $C\Dx\Dt$ in one timestep by solving the nonlinear
system only approximately, \cs{it suffices} to perform
$\mathcal{O}(\log(\Dx^{-1}\Dt^{-1}))$ \cs{many Newton} iterations. 
If we take
$\Dt=\theta \Dx$ for some constant $\theta>0$, these are
\cs{altogether} 
$\mathcal{O}(\log(\Dx^{-2}))=\mathcal{O}(\log(\Dx^{-1}))$ 
\cs{Newton} steps. 
In each step of the Newton iteration, we invert and multiply a
tridiagonal matrix of size $\mathcal{O}(\Dx^{-2})$ with a vector of
length $\mathcal{O}(\Dx^{-1})$ and subtract it from another vector of
length $\mathcal{O}(\Dx^{-1})$. The tridiagonal matrix can be inverted
in $\mathcal{O}(\Dx^{-1})$ operations using the Thomas algorithm (in
case of periodic boundary conditions we use the Sherman-Morrison
formula). 
Hence the total number of floating point operations 
which are  necessary for one Newton step is
$\mathcal{O}(\Dx^{-1})$. It follows that the work done in one
timestep is of order $\mathcal{O}(\log(\Dx^{-1})\Dx^{-1})$. 
As there are altogether $n=T/\Dt$ timesteps, and 
since we can choose the timestep of order
$\Dt=\theta\Dx$, we obtain the following bound on the total work
for one execution of the implicit scheme,
\begin{equation*}
  W_{\Delta}^{\mathrm{im}} = \mathcal{O}(\Dx^{-2}\log(\Dx^{-1})
\;.
\end{equation*}
In the Monte Carlo Finite Difference Methods (MC-FDMs), we combine MC
sampling of the random initial data with the FDMs \eqref{eq:IMP} and
\eqref{eq:EXP}. In the convergence analysis of these schemes, we shall
require the application of the FDMs \eqref{eq:IMP} and \eqref{eq:EXP}
to random initial data, flux function and diffusion operator
$(u_0,f,A) \in L^p(\Omega; E_1)$ for some $1\le p\le\infty$.  Given a
draw $(u_0(\omega;\cdot),f(\omega;\cdot),A(\omega;\cdot))$ of
$(u_0,f,A)$, the FDMs \eqref{eq:ini} with \eqref{eq:EXP} or
\eqref{eq:IMP} define families $u_\Delta(\omega;x,t)$ of grid
functions. We have the following
\begin{proposition}\label{prop4.3}
  Consider the FDMs \eqref{eq:ini}--\eqref{eq:EXP}, \eqref{eq:IMP} for
  the approximation of the entropy solution corresponding to the draw
  $(u_0,f,A)(\omega)$ of the random data.

  Then, the random grid functions $\Omega \ni \omega \longmapsto
  u_\Delta (\omega;x,t)$ defined by \eqref{eq:pwc} satisfy, for every
  $0 < \ov{t} < \infty$, $0 < \Delta x < 1$, and every $k \in \IN \cup
  \{\infty\}$ the stability bounds:
  \begin{equation*}
    \norm{u_\Delta(\cdot;\cdot, \ov{t})}_{L^k(\Omega; L^\infty(\R))} 
    \le 
    \norm{u_0}_{L^k(\Omega; L^\infty(\R))}, 
  \end{equation*}
  \begin{equation*}
    \norm{u_\Delta(\cdot;\cdot, \ov{t})}_{L^k(\Omega; L^1(\R))} 
    \le 
    \norm{u_0}_{L^k(\Omega; L^1(\R))}.
  \end{equation*}
  We also have the bound
  \begin{multline}\label{4.22d}
    \norm{u(\cdot;\cdot, \ov{t}) - u_\Delta(\cdot;\cdot,
      \ov{t})}_{L^k(\Omega; L^1(\R))} \le
    \norm{u_0 - u_\Delta(\cdot;\cdot,0)}_{L^k(\Omega; L^1(\R))}\\
    + C(1+\ov{t}) \Delta x^{1/3}
    \Bigl\{\norm{\left(1+\norm{f(\omega;\cdot)}_{\mathrm{Lip}}\right)
    \abs{(u_0)(\omega)}_{BV(\mathbb{R})}}_{L^k(\Omega)}\\
    +\norm{A(u_0)_x}_{L^k(\Omega;L^1(\R))}
    +\norm{\abs{(f(u_0)-A(u_0)_x)(\omega)}_{BV(\R)}}_{L^k(\Omega)}\Bigr\}.
  \end{multline}
\end{proposition}
\begin{remark}\label{rem:conseq}
  Under the assumptions \eqref{eq:assu} -- \eqref{eq:assAf},
  \eqref{4.22d} becomes
  \begin{multline*}
    \norm{u(\cdot;\cdot, \ov{t}) - u_\Delta(\cdot;\cdot,
      \ov{t})}_{L^k(\Omega; L^1(\R))} \le \norm{u_0 -
      u_\Delta(\cdot;\cdot,0)}_{L^k(\Omega; L^1(\R))}
    \\
    + C(1+\ov{t}) \Delta x^{1/3} \seq{\left(1+C_f+C_A\right)
      C_{\mathrm{TV}} +C_{A,f}}.
  \end{multline*}
\end{remark}
\begin{remark}\label{rem:ass}
  We see from \eqref{4.22d} that in order to obtain the convergence
  rate of $1/3$ in $L^k(\Omega)$ it would \cs{suffice} to assume
  \eqref{3.16}, \eqref{eq:assu}, \eqref{eq:assA},
  $\abs{f(u_0)-A(u_0)_x}_{BV(\R)}(\omega)\in L^k(\Omega)$, $A', f'\in
  L^{pk}(\Omega)$, $\abs{u_0}_{BV(\R)}\in L^{qk}(\Omega)$ for some
  $p,q\ge 1$ satisfying $1/p+1/q=1$. However, in order to obtain a
  uniform CFL-condition for the explicit scheme (which gives us the
  same asymptotic work estimate for each simulation with the explicit
  scheme), we need \eqref{eq:assf} and \eqref{eq:assA2} to hold as
  well.
\end{remark}

\subsection{MC-FDM Scheme}
We next define and analyze the MC-FDM scheme.  It is based on the
straightforward idea of generating, possibly in parallel, independent
samples of the random initial data and then, for each sample of the
random initial data, flux function and diffusion operator, to perform
one FD simulation.  The error of this procedure is bound by two
contributions: a (statistical) sampling error and a (deterministic)
discretization error. We express the asymptotic efficiency of this
approach (in terms of overall error versus work). It will be seen that
the efficiency of the MC-FDM is, in general, inferior to that of the
deterministic schemes \eqref{eq:EXP} and \eqref{eq:IMP}.  The present
analysis will constitute a key technical tool in our subsequent
development and analysis of the multilevel MC-FDM (``MLMC-FDM'' for
short) which does not suffer from this drawback.

\subsubsection{Definition of the MC-FDM Scheme}
We consider once more the initial value problem \eqref{eq:RSCL} with
random data $(u_0,f,A)$ satisfying \eqref{eq:assu} -- \eqref{eq:assAf}
and \eqref{3.16} for sufficiently large $k \in \IN$ (to be specified
in the convergence analysis).  The MC-FDM scheme for the MC estimation
of the mean of the random entropy solutions then consists in the
following:

\begin{definition}\label{def4.4} (MC-FDM Scheme)
  Given $M \in \IN$, generate $M$ i.i.d.~samples
  $\{(\wh{u}_0^i,\wh{f}^i,\wh{A}^i)\}^M_{i=1}$.  Let
  $\{\wh{u}^i(\cdot,t)\}_{i=1}^M$ denote the unique entropy solutions
  of the degenerate convection diffusion equations \eqref{eq:main1} for
  these data samples, i.e.
  \begin{equation*}
    \wh{u}^i(\cdot,t) = S(t) \left(\wh{u}_0^i,\wh{f}^i,\wh{A}^i\right), \quad i
    = 1,\ldots, M. 
  \end{equation*}
  Then the MC-FDM approximations of $\cM^k(u(\cdot,t))$ are defined as
  statistical estimates from the ensemble
  \begin{equation*}
    \{\wh{u}_\Delta^i (\cdot,t)\}^M_{i=1}
  \end{equation*}
  obtained from the FD approximations by \eqref{eq:EXP} or
  \eqref{eq:IMP} of \eqref{eq:main1} with data samples
  $\{(\wh{u}_0^i,\wh{f}^i,\wh{A}^i)\}_{i=1}^M$: Specifically, the
  first moment of the random solution $u(\omega;\cdot, t)$ at time $t
  > 0$, is estimated as
  \begin{equation}\label{4.25}
    \cM^1(u(\cdot, t))  \approx E_M[u_\Delta (\cdot, t)] 
    := 
    \dis\frac{1}{M} \;\dis\sum\limits^M_{i=1} \,\wh{u}_\Delta^i (\cdot,t)\,,
  \end{equation}
  and, for $k > 1$, the $k$th moment (or $k$-point correlation
  function) $\cM^k(u(\cdot,t))=\IE[(u(\cdot,t))^{(k)}]$ is estimated
  by
  \begin{equation}\label{4.26}
    E_M^{(k)} [u_\Delta (\cdot, t)] 
    := 
    \frac{1}{M} \;\dis\sum\limits^M_{i=1} \,
    \underbrace{(\wh{u}_\Delta^i \otimes \dots \otimes
      \wh{u}_\Delta^i)}_{k\text{-times}}  
    \,(\cdot,t)\,.
  \end{equation}
  More generally, for $k > 1$, we consider time instances
  $t_1,\dots,t_k \in (0,T]$, $T < \infty$, and define the statistical
  FDM estimate of $\cM^k(u)(t_1,...,t_k)$ by
  \begin{equation}\label{4.27}
    E^{(k)}_M \,[u_\Delta] \,(t_1,\dots,t_k)
    : = 
    \dis\frac{1}{M} \sum\limits^M_{i=1} 
    \underbrace{
      (\wh{u}_\Delta^i (\cdot,t_1) \otimes \dots \otimes
      \wh{u}_\Delta^i (\cdot,t_k))}_{k\text{-times}}. 
  \end{equation}
\end{definition}

\subsubsection{Convergence Analysis of MC-FDM}
We next address the convergence of $E_M [u_\Delta]$ to the mean
$\IE(u)$. Arguing as in \cite{Mishr478,MishSch10a}, 
and using the error bounds in Proposition \ref{prop4.3},
we obtain the following result.
\begin{theorem}\label{theo4.5}
  Assume that
  \begin{equation*}
    u_0 \in L^2(\Omega, L^1(\R))
  \end{equation*}
  and that \eqref{eq:assu} -- \eqref{eq:assAf}
  hold. 
  Then the MC estimate $E_M[u_\Delta (\cdot,t)]$ defined in 
    \eqref{4.25}) as in Definition~\ref{def4.4} satisfies, for every
  $M$, the error bound
  \begin{multline}\label{4.29}
    \norm{\IE[u(\cdot,t)] -
    E_M[u_\Delta(\cdot,t;\omega)]}_{L^2(\Omega;L^1(\R))} \le
    C\biggl\{
    M^{-1/2} \norm{u_0}_{L^2(\Omega;L^1(\R))}\\
    + \norm{u_0 - u_\Delta(\cdot;\cdot,0)}_{L^2(\Omega;L^1(\R))} +
    \Delta x^{1/3} (1+\ov{t})
    \Bigl\{\norm{A(u_0)_x}_{L^2(\Omega;L^1(\R))}\\
    +\norm{\left(1+\norm{f(\omega;\cdot)}_{\mathrm{Lip}}\right)
      \abs{u_0}_{BV(\R)}(\omega)}_{L^2(\Omega)} 
    +\norm{\abs{f(u_0)-A(u_0)_x}_{BV(\R)}(\omega)}_{L^2(\Omega)}\Bigr\}\biggr\}.
  \end{multline}
  where $C > 0$ is independent of $M$ and of $\Delta x$ as $M \rt
  \infty$ and as $\Delta x,\Delta t \downarrow 0$.
\end{theorem}

\subsubsection{Work estimates}\label{sssec:workmc}
We have seen in Sections~\ref{ssec:workexp} and \ref{ssec:workimp}
that the computational work to obtain $\{u_\Delta(\cdot,t)\}_{0\le
  t\le T}$, computed by the explicit or implicit scheme respectively,
is asymptotically, as $\Dx,\Dt\rightarrow 0$, of order
\begin{equation*}
  W^{\rm{ex}}_\Delta \leq C \Dx^{-11/3},\quad  W^{\rm{im}}_\Delta\leq C
  \Dx^{-2} \log(\Dx^{-1}),  
\end{equation*}
which implies that the work for the computation of the 
MC estimate $E_M[u_\Delta(\cdot,t)]$ 
is of order
\begin{equation}\label{4.31}
  W^{\rm{ex}}_{\Delta,M} \leq C M\Dx^{-11/3},\quad
  W^{\rm{im}}_{\Delta,M} \leq C M\Dx^{-2} \log(\Dx^{-1}),  
\end{equation}
so that we obtain from (\ref{4.29}) the convergence order in terms of
work: To this end we equilibrate in (\ref{4.29}) the two bounds by
choosing $M^{-1/2} \sim \Delta x^{1/3}$, i.e. $M = C \Delta x^{-2/3}$.
Inserting in \eqref{4.31} yields
\begin{equation*}
  W^{\rm{ex}}_{\Delta,M} \leq C \Dx^{-13/3},\quad
  W^{\rm{im}}_{\Delta,M} \leq C \Dx^{-8/3} \log(\Dx^{-1}), 
\end{equation*}
so that we obtain from \eqref{4.29}
\begin{subequations}\label{4.33}
  \begin{align}\label{seq1:explw}
    \norm{\IE[u(\cdot,t)] - E_M [u_\Delta(\cdot,t)]
    }_{L^2(\Omega;L^1(\R))} &\le \CutAf \Delta x^{1/3} \le
     \CutAf\,(W^{\rm{ex}}_{\Delta,M})^{-1/13},\\
    \norm{\IE[u(\cdot,t)] - E_M [u_\Delta(\cdot,t)]
    }_{L^2(\Omega;L^1(\R))} &\le
     \CutAf \,(W^{\rm{im}}_{\Delta,M}(\log(W^{\rm{im}}_{\Delta,M}))^{-1})^{-1/8},
    \label{seq2:implw}
  \end{align}
\end{subequations}
where $\CutAf$ is given by
\begin{multline}\label{eq:const}
   \CutAf = C(1+t)\Bigl\{\|\left(1+
    \|f(\omega;\cdot)\|_{\mathrm{Lip}}\right)\abs{u_0}_{BV(\R)}(\omega)\|_{L^2(\Omega)}\\ 
  +\|A(u_0)_x\|_{L^2(\Omega;L^1(\R))}
  +\|\abs{f(u_0)-A(u_0)_x}_{BV(\R)}(\omega)\|_{L^2(\Omega)}\Bigr\}.
\end{multline}
On the other hand, in the deterministic case we have the convergence rates,
\begin{subequations}\label{eq:wratedet}
  \begin{align}\label{seq1:wratedet}
    \norm{u(\cdot,t)-u_\Delta(\cdot,t)}_{L^1(\R)} &\le C_T\Delta x^{1/3}
    \le
    C_T\,(W^{\rm{ex}}_{\Delta})^{-1/11},\\
    \label{seq2:wratedet}
    \norm{u(\cdot,t) -u_\Delta(\cdot,t)}_{L^1(\R)} &\le
    C_T\,(W^{\rm{im}}_{\Delta}(\log(W^{\rm{im}}_{\Delta}))^{-1})^{-1/6},
  \end{align}
\end{subequations}
with respect to work.

\subsection{Multilevel MC-FDM}
We next present and analyze a scheme that allows us to achieve almost
the accuracy versus work bound \eqref{eq:wratedet} of the
deterministic FDM also for the stochastic data $(u_0,f,A)$, rather
than the single level MC-FDM error bound (\ref{4.33}).  The key
ingredient in the Multilevel Monte Carlo Finite Difference (MLMC-FDM)
scheme is simultaneous MC sampling on different levels of resolution
of the FDM, {\em with level dependent numbers $M_\ell$ of MC samples}.
To define these, we introduce some notation.

\subsubsection{Notation}
The MLMC-FDM is defined as a {\em multilevel discretization} in $x$ and $t$ 
with level dependent numbers $M_\ell$ of samples. 
To this end, we assume we are given a family  
 of nested grids with cell sizes
 \begin{equation}\label{4.35}
   \Delta x_\ell=2^{-K\ell}\Dx_0,\quad  \ell \in \IN_0,
 \end{equation}
 for some $\Dx_0>0$, $K$ such that $2^{K}\in \IN\setminus
 \{0,1\}$. Similarly, we denote,
\begin{equation*}
\Dt_\ell=C \Dx^{8/3}_\ell,
\end{equation*}
the size of the time step for the explicit scheme corresponding to
grid size $\Dx_\ell$ and
\begin{equation*}
\Dt_\ell= \theta \Dx_\ell,
\end{equation*}
the size of the time step for the implicit scheme at level $\ell$. We
denote by $u_{\ell}$ the
approximation to \eqref{eq:1dmain} computed by \eqref{eq:EXP} or
\eqref{eq:IMP} on the grid with cell and time step size
$\Delta_\ell:=( \Dx_\ell, \Dt_\ell)$.

\subsubsection{Derivation of MLMC-FDM}
As in plain MC-FDM, our aim is to estimate, for $0 < t < \infty$, the
expectation (or ``ensemble average'') $\IE[u(\cdot,t)]$ of the random
entropy solution of \eqref{eq:RSCL} with random data
$(u_0,f,A)(\omega)$, $\omega \in \Omega$, satisfying \eqref{eq:rv} --
\eqref{3.16} for sufficiently large values of $k$ (to be specified in
the sequel).  As in the previous section, $\IE[u(\cdot,t)]$ will be
estimated by replacing $u(\cdot,t)$ by a FDM approximation.

We generate a sequence of approximations,
$\{u_\ell(\cdot,t)\}^\infty_{\ell = 0}$ on the nested meshes with cell
sizes $\Dx_\ell$, time steps of sizes $\Delta t_\ell$.  In the
following we set $u_{-1}(\cdot,t) : = 0$.  Then, given a {target} level $L \in
\IN$ of spatial resolution, we have
\begin{equation}\label{4.40}
  \IE[u_L(\cdot,t)] 
  = 
  \IE\Big[ \sum\limits^L_{\ell = 0} \,(u_\ell (\cdot,t) - u_{\ell -
    1}(\cdot,t))\Big] . 
\end{equation}
We next estimate each term in (\ref{4.40}) statistically by a 
MCM with a level-dependent number of samples, $M_\ell$; 
this gives the MLMC-FDM estimator
\begin{equation}\label{4.41}
  E^L[u(\cdot,t)] 
  = 
  \sum\limits^L_{\ell = 0} \,
  E_{M_\ell} [u_\ell (\cdot,t) - u_{\ell -1}(\cdot,t)]
\end{equation}
where $E_M[u_\Delta(\cdot,t)]$ is as in (\ref{4.25}), and where
$u_\ell(\cdot,t)$ is computed on the mesh with grid size $\Dx_\ell$
and time step $\Dt_\ell$.

Statistical moments $\cM^k(u)(t_1,...,t_k)$ of order $k\ge 2$
(resp. the $k$-th order space-time correlation functions) of the
random entropy solution $u$ can be estimated in the same way: based on
\eqref{4.26} in Definition \ref{def4.4}, the straightforward
generalization along the lines of the MLMC estimate \eqref{4.41} of
the MC estimate \eqref{4.27} for $\cM^k(u)(t)$ leads to the definition
of the MLMC-FDM estimator
\begin{equation} \label{eq:MLMCFTP} 
  E^{L,(k)}[u(\cdot,t)] :=
  \sum\limits^L_{\ell = 0} \,E_{M_\ell} [(u_\ell(\cdot,t))^{(k)} -
  (u_{\ell -1}(\cdot,t))^{(k)}]\;, \quad 0<t<\infty .
\end{equation}
This generalizes \eqref{4.41} to moments $\cM^k(u)(t)$ of order $k>1$.
\footnote{We assume here for notational convenience that
          $t_1 = t_2 = ... = t_k = t$.
          This implies that our $k$-th moment estimate
          only requires access to the FDM solutions at
          time $t$.
          The following developments directly generalize
          to the analysis of $k$-point
          temporal correlation functions
          of the random entropy solution as well; in this
          case, however, access to the full history of
          FDM solutions $v_\ell(\cdot,t)$ for $0\le t \le T < \infty$
          is required \cs{for the MC estimation of these correlations}.}

\subsubsection{Convergence Analysis}
We first analyze the MLMC-FDM mean field error
\begin{equation}\label{4.42}
  \norm{ \IE[u(\cdot,t)] - E^L[u(\cdot,t)] }_{L^2(\Omega;L^1(\R^d))}
\end{equation}
for $0 < t < \infty$ and $L \in \IN$.  In particular, we are
interested in the choice of the sample sizes $\{M_\ell\}^\infty_{\ell
  =0}$ such that, for every $L \in \IN$, the MLMC error (\ref{4.42})
is of order $(\Delta x_L)^{1/3}$.  The principal issue in the design
of MLMC-FDM is the optimal choice of $\{M_\ell\}^\infty_{\ell =0}$
such that, for each $L$, an error \eqref{4.42} is achieved with
minimal total work given by (based on \eqref{4.31}),
\begin{subequations}\label{4.43}
  \begin{align}\label{4.43a}
    W_{L,MLMC}^{\rm{ex}} &= C\dis\sum\limits^L_{\ell = 0}\,M_\ell
    W^{\rm{ex}}_{\Delta_\ell}
    =\mathcal{O}\left( \dis\sum\limits^L_{\ell = 0}\,M_\ell \Delta x_\ell^{-11/3}\right),\\
    \label{4.43b}
    W_{L,MLMC}^{\rm{im}} &= C\dis\sum\limits^L_{\ell = 0}\,M_\ell
    W^{\rm{im}}_{\Delta_\ell} =\mathcal{O}\left( \dis\sum\limits^L_{\ell = 0}\,M_\ell
    \Delta x_\ell^{-2}\log(\Dx_\ell^{-1})\right).
  \end{align}
\end{subequations}
To estimate \eqref{4.42}, we write (recall that $u_{-1} : = 0$)
using the triangle inequality, the linearity of the mathematical
expectation $\IE[\cdot]$ and the definition \eqref{4.41} of the 
MLMC estimator
\begin{align*}
  \|\IE[u(\cdot,t)] - &E^L[u(\cdot,t)]\|_{L^2(\Omega;L^1(\R))} \\
  &\le
  \norm{\IE[u(\cdot,t)] - \IE[u_L(\cdot,t)]}_{L^2(\Omega;L^1(\R))}
  + \norm{\IE[u_L(\cdot,t)] -
    E^L[u(\cdot,t)]}_{L^2(\Omega;L^1(\R))}
  \\
  &= \norm{\IE[u(\cdot,t)] - \IE[u_L (\cdot,t)]}_{L^2(\Omega;L^1(\R))}
  \\
  &\qquad +
  \Bigl\| \dis\sum\limits^L_{\ell = 0} \, \IE[u_\ell - u_{\ell - 1}] -
    E_{M_\ell} [u_\ell - u_{\ell - 1}] \Bigr\|_{L^2(\Omega;L^1(\R))}
  \\
  &=:\text{I}+\text{II}
\end{align*}
We estimate terms I and II separately.  
By linearity of the expectation, \cs{term} I equals
\begin{equation*}
  \text{I}
  =
  \norm{\IE[u(\cdot,t) - u_L(\cdot,t)]}_{L^1(\R)}
  = 
  \norm{u(\cdot,t) - u_L(\cdot,t)}_{L^1(\Omega;L^1(\R))}
\end{equation*}
which can be bounded by \eqref{4.22d} with $k=1$. 
We hence focus on \cs{term} II, i.e.,
\begin{align*}
  \text{II} &\le \sum\limits^L_{\ell = 0}
  \norm{\IE[(u_\ell-u_{\ell-1})(\cdot,t)]-E_{M_\ell}
  [(u_\ell-u_{\ell-1})(\cdot,t)]}_{L^2(\Omega;L^1(\R))}\\ 
  & \stackrel{\eqref{4.6}}{\le} \sum\limits^L_{\ell = 0}
  M_\ell^{-\frac{1}{2}}
  \Bigl(\int_\Omega \|u_\ell (\cdot,t; \omega) - u_{\ell -1} (\cdot,t;
  \omega)\|^2_{L^1(\R)} \,d\IP(\omega)\Bigr)^{\frac{1}{2}}\\ 
  & =\sum\limits^L_{\ell = 0} \,M_\ell^{-\frac{1}{2}}
  \norm{u_\ell(\cdot,t) - u_{\ell - 1}(\cdot,t)}_{L^2(\Omega; L^1(\R))}.
\end{align*}
We estimate for every $\ell \ge 0$ the size of the 
detail $u_\ell - u_{\ell-1}$ with the triangle inequality 
\begin{equation*}
  \norm{u_\ell(\cdot,t) - u_{\ell -1}(\cdot, t)}_{L^2(\Omega;
    L^1(\R))} \le \norm{u(\cdot,t) - u_\ell(\cdot,t)}_{L^2(\Omega;
    L^1(\R))}  + \norm{u(\cdot,t) - u_{\ell -
      1}(\cdot,t)}_{L^2(\Omega; L^1(\R))}.
\end{equation*}
Using here (\ref{4.22d}) with $\ov{t} = t$, $k = 2$, \eqref{eq:const}
and \eqref{4.35}, we obtain for every $\ell \in \IN$ the estimate
\begin{align*}
  \norm{(u_\ell - u_{\ell -1})(\cdot, t)}_{L^2(\Omega; L^1(\R))} &\le
  \norm{u_0 - u_\ell(\cdot;\cdot,0)}_{L^2(\Omega; L^1(\R))}+
  \norm{u_0 - u_{\ell - 1}(\cdot;\cdot,0)}_{L^2(\Omega; L^1(\R))}\\
  &\qquad +  \CutAf \, (1+2^{K/3})\, \Delta x_\ell^{1/3}.
\end{align*}
Using that for $0 \le s \le 1$, the cell-averages
$u_\ell(\cdot;\cdot,0)$ satisfy, for every $k \in \IN$ and for every
$1 \le q \le \infty$,
\begin{equation*}
\norm{u_0 - u_\ell(\cdot;\cdot,0)}_{L^k (\Omega; L^q(\R))} 
\le 
C \Delta x_\ell^s \norm{u_0}_{L^k(\Omega; W^{s,q}(\R))}\,,
\end{equation*}
we arrive at the error bound
\begin{equation*}
  \norm{u_\ell(\cdot,t) - u_{\ell -1}(\cdot, t)}_{L^2(\Omega; L^1(\R))} 
  \le
  \left(  \CutAf \left(1+2^{K/3}\right)+ 
    C \Dx_\ell^{2/3}\norm{\abs{u_0}_{BV(\R)}}_{L^2(\Omega)}\right) 
  \Delta x_\ell^{1/3}.
\end{equation*}
Summing this error bound over all discretization levels $\ell=0,...,L$,
we prove the main result of the present paper.
\begin{theorem}\label{theo4.7}
  Assume \eqref{eq:assu} -- \eqref{3.16} for some $k\ge 2$ and
  \eqref{4.35}.  Then, for any sequence $\{M_\ell\}^\infty_{\ell = 0}$
  of sample sizes at mesh level $\ell$, we have for the MLMC-FDM
  estimate $E^L[u(\cdot,t)]$ in \eqref{4.41} the error bound
  \begin{equation}\label{4.47}
    \begin{aligned}
      \Bigl\|\IE[u(\cdot,t)] - &E^L[u(\cdot,t)]\Bigr\|_{L^2(\Omega;L^1(\R))}\\
      &\le C \left\{  \CutAf^1 \Delta x_L^{1/3} + \Delta x_L
        \norm{\abs{u_0}_{BV(\R)}}_{L^1 (\Omega)} \right\}
      \\
      &\quad + C\left\{\sum\limits^L_{\ell = 0} M_\ell^{-1/2} \Delta
        x_\ell^{1/3}\right\} \left( \CutAf^2
        (1+2^{K/3})+\Dx_\ell^{2/3}
        \norm{\abs{u_0}_{BV(\R)}}_{L^2(\Omega)}\right)
    \end{aligned}
  \end{equation}
  where we have denoted
  \begin{multline*}
     \CutAf^j = C(1+t)\Bigl\{
    \norm{\left(1+\norm{f(\cdot;\cdot)}_{\mathrm{Lip}}\right)
      \abs{u_0}_{BV(\R)}}_{L^j(\Omega)}
    \\
    +\norm{A(u_0)_x}_{L^j(\Omega;L^1(\R))}
    +\norm{\abs{f(u_0)-A(u_0)_x}_{BV(\R)}}_{L^j(\Omega)}\Bigr\}.
  \end{multline*}
  $j=1,2$ and $C>0$ is a constant that is independent of the parameters 
  \cs{ $u_0$, $f$, $t$ and $A$.}
\end{theorem}
\cs{
The upper bound obtained in}
Theorem~\ref{theo4.7} is the basis for an optimization of the numbers
$M_\ell$ of MC samples across the mesh levels. 
Our selection of the
level dependent Monte Carlo sample sizes $M_\ell$ will be based on the
last term in the error bound \eqref{4.47}; we select in \eqref{4.47}
the $M_\ell$ such that as $\Delta \downarrow 0$, all terms equal the
error estimate $\Delta x_L^{1/3}$ at the finest level $L$.  This
motivates choosing $M_\ell$ such that
\begin{equation*}
  M_\ell^{-\frac{1}{2}}   \Delta x_\ell^{1/3} = \hat{C}\Delta x_L^{1/3}, 
\quad
\ell = 0,\ldots , L-1\,.
\end{equation*}
Here, $\hat{C}$ is some positive integer that is independent of
$\ell$, $L$.  
Using 
\begin{equation*}
  \Delta x_\ell = 2^{-\ell K}\Dx_0, \quad \ell = 0,1,2,\ldots,
\end{equation*}
we find 
$M_\ell = \hat{C}\Delta x_\ell^{2/3} \Delta x_L^{-2/3} = \widetilde{C}\, 2^{2K(L - \ell)/3}$.  
This implies in \eqref{4.47} the bound
\begin{equation}\label{4.49}
  \norm{\IE[u(\cdot,t)] - E^L[u(\cdot,t)]}_{L^2(\Omega;L^1(\R))} 
  \le (L+1)(1+2^{K/3}) \whCutAf \Delta x_L^{1/3},
\end{equation}
where 
$\whCutAf = \max_{j\in\{1,2\}}\{
 \CutAf^j+\|\abs{u_0}_{BV(\R)}\|_{L^j(\Omega)}\}$, 
while the total cost is, using \eqref{4.43}, bounded by
\begin{subequations}\label{4.50}
  \begin{align}\label{eq:4.50a}
    W_{L,MLMC}^{\rm{ex}} &\leq
    C\sum\limits^L_{\ell = 0} M_\ell \Dx_\ell^{-11/3} 
    =C\sum\limits^L_{\ell = 0} 2^{2KL/3+3\ell K}
    =C 2^{11 K L/3}= \mathcal{O}\left(\Dx_L^{-11/3}\right),
    \\
    \label{4.43c}
    W_{L,MLMC}^{\rm{im}} &\leq
    C \sum\limits^L_{\ell = 0} M_\ell \Delta
    x_\ell^{-2}\log(\Dx_\ell^{-1})
    =CK \log(2) 2^{2KL/3} \sum\limits^L_{\ell = 0} \ell 2^{4K\ell/3}
    \\
    &= C 2^{2KL}\log(2^{KL})=\mathcal{O}\left( \Dx_L^{-2}\log(\Dx_L^{-1})\right)\notag.
  \end{align}
\end{subequations}
We observe that this is asymptotically the same work as the one needed
for one deterministic approximation of \eqref{eq:1dmain} using
\eqref{eq:EXP} or \eqref{eq:IMP} with grid size $\Dx_L$ and
corresponding time step
$\Dt_L$.

Inserting \eqref{4.50} into the asymptotic error bound \eqref{4.49},
we obtain the following error estimate in terms of work
\begin{subequations}\label{eq:MLMCFVMerrWork}
  \begin{align}\label{eq:mlmcwexp}
    &\norm{\IE[u(\cdot,t)] - E^L[u(\cdot,t)]}_{L^2(\Omega;L^1(\R))}
    \le (L+1)\left(1+2^{K/3}\right) \whCutAf
    (W^{\rm{ex}}_{L,MLMC})^{-1/11},
    \\ 
    \label{eq:mlmcwimp}
    &\norm{\IE[u(\cdot,t)] - E^L[u(\cdot,t)]}_{L^2(\Omega;L^1(\R))} 
    \\
    &\qquad\le (L+1)\left(1+2^{K/3}\right)
    \whCutAf (W^{\rm{im}}_{L,MLMC}
    (\log(W^{\rm{im}}_{L,MLMC}))^{-1})^{-1/6}\notag. 
  \end{align}
\end{subequations}
%
We observe that 
the MLMC-FDM (\ref{4.50}) behaves, 
in terms of accuracy versus work,
as $L \rt \infty$, as the 
deterministic FDM up to $\log$-terms,
where the error vs. work was estimated in \eqref{eq:wratedet}. 
Now one can balance $L$ and $K$ in order to obtain as small a 
constant as possible.

\section{Numerical Experiments}
\label{sec:NumExp}
In this section, we will test the method on some numerical examples
from two-phase flow in porous media. In one space dimension, the time 
evolution of the water saturation $s:=s^w\in [0,1]$ can be modeled by
the conservation law 
\begin{align}\label{eq:sw}
  \begin{split}
    s_t+f(s)_x=(a(s)s_x)_x,&\quad (t,x)\in [0,T]\times D,\\
    s(0,x)=s_0(x),&\quad x\in D,
  \end{split}
\end{align}
where $D\in \R$ is a bounded interval, 
$f$ and $a$ are of the form
\begin{equation}\label{eq:fa}
  f(s)=q\frac{\lambda^w(s)}{\lambda^w(s)+\lambda^o(s)},\quad a(s)=\nu
  \overline{K}\frac{\lambda^w(s)\lambda^o(s)}{\lambda^w(s)+\lambda^o(s)} p_c'(s) 
\end{equation}
where $q$ denotes total flow rate, 
$\overline{K}$ the rock permeability (we will set
$\overline{K}=q=1$ for simplicity), 
$\nu$ a small number, and $p_c$ the capillary
pressure for which we will use the expression
\begin{equation*}
  p_c(s)=-\left(s^{-4/3}-1\right)^{1/4},
\end{equation*}
which is taken from \cite{cappressure}, and $\lambda^w$, $\lambda^o$ are
the phase mobilities/relative permeabilities of the water and the oil
phase respectively. The relative permeability of the water phase
$\lambda^w$ is a monotone function with $\lambda^w(0)=0$,
$\lambda^w(1)=1$, and the relative permeability of the oileic phase
$\lambda^o$ is a monotone decreasing function such that
$\lambda^o(0)=1$ and $\lambda^o(1)=0$. Often one uses the simple expressions
\begin{equation*}
  \lambda^w(s)=s^2,\quad \lambda^o(s)=(1-s)^2.
\end{equation*}
Such a form of the relative permeability is of course a
simplification, and more accurate models are based on experiments, and
these functions therefore have some uncertainty associated with them.
Hence it is natural to model the relative permeabilities as random
variables.
Equations \eqref{eq:sw} have to be augmented with suitable boundary
conditions. 
In the ensuing numerical experiments,
we use the domains $D = (0,2)$ and $D=(0,0.5)$ and
periodic boundary conditions, in order to avoid issues
related to unbounded domains or to boundary effects.

\subsection{Random exponent}
\label{sec:RndExp}
For this example we will model the relative permeabilities by
\begin{equation}\label{eq:relpermexp}
  \lambda^w(s)=|s|^{p(\omega)},\quad
  \lambda^o(s)=|1-s|^{p(\omega)},
\end{equation}
where the random exponent $p$ is uniformly distributed in the interval
$[1.5,2.5]$. As initial data, we use
\begin{equation}\label{eq:u02}
  s_0(x)=\begin{cases}
    0.1, & x\in [0,0.1)\cup [1,2),\\
    0.8, & x\in [0.1,1),
  \end{cases}
\end{equation}
and periodically extended outside $[0,2]$.  Figure~\ref{fig:eg1} shows
a sample $s(\omega;t,\cdot)$ of the random entropy solution at time
$T=0.3$, and an estimate of the mean $\mathbb{E}[s(\cdot,0.3)]$
computed by the explicit multilevel Monte Carlo finite difference
method with $M_0=8$, $L=8$, $\Delta x_0=2^{-3}$, $K=1$ and CFL-number
$0.4$.
\begin{figure}[ht]
  \begin{tabular}{lr}
    \includegraphics[width=0.49\textwidth]{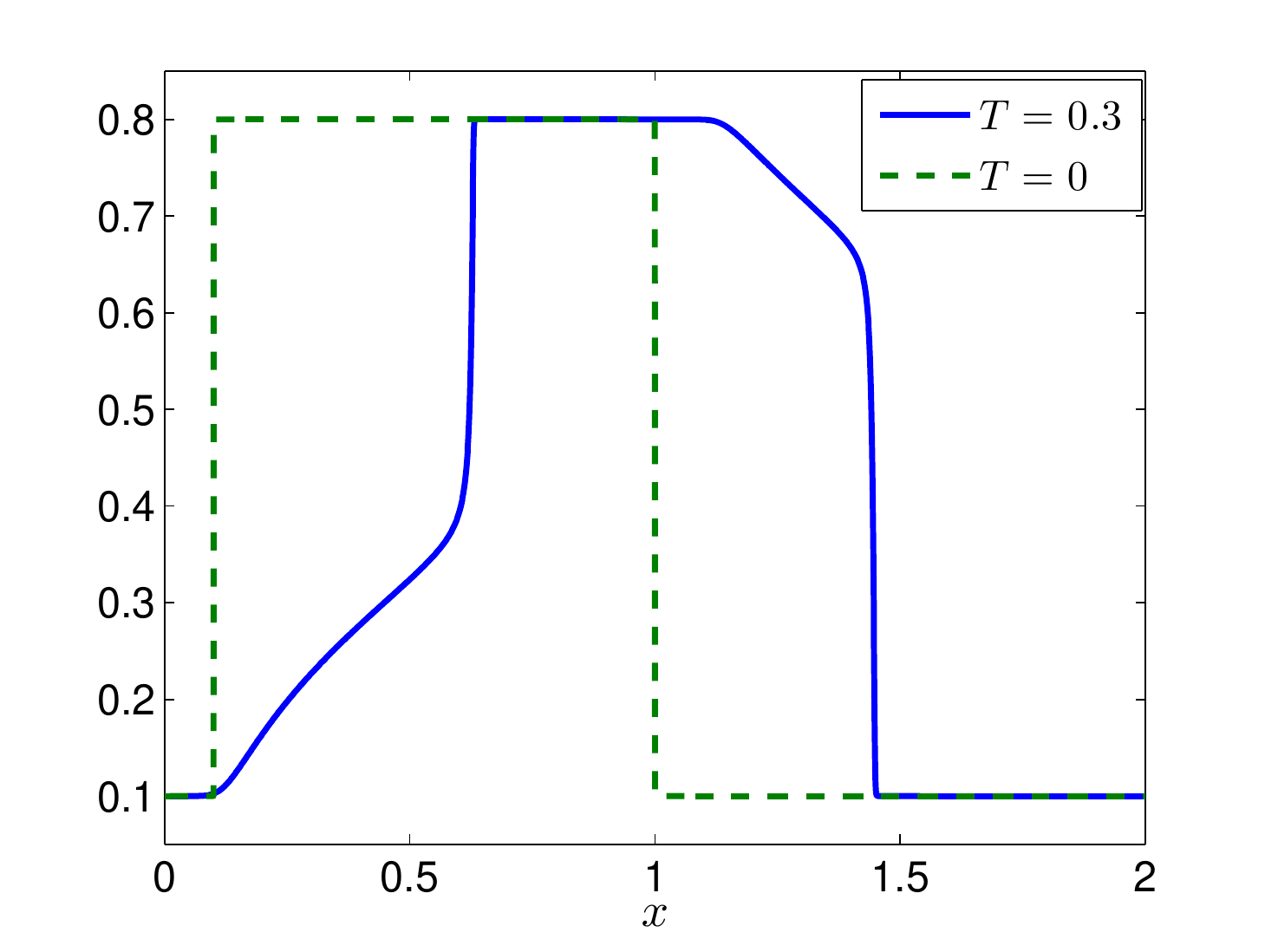}
    &\includegraphics[width=0.49\textwidth]{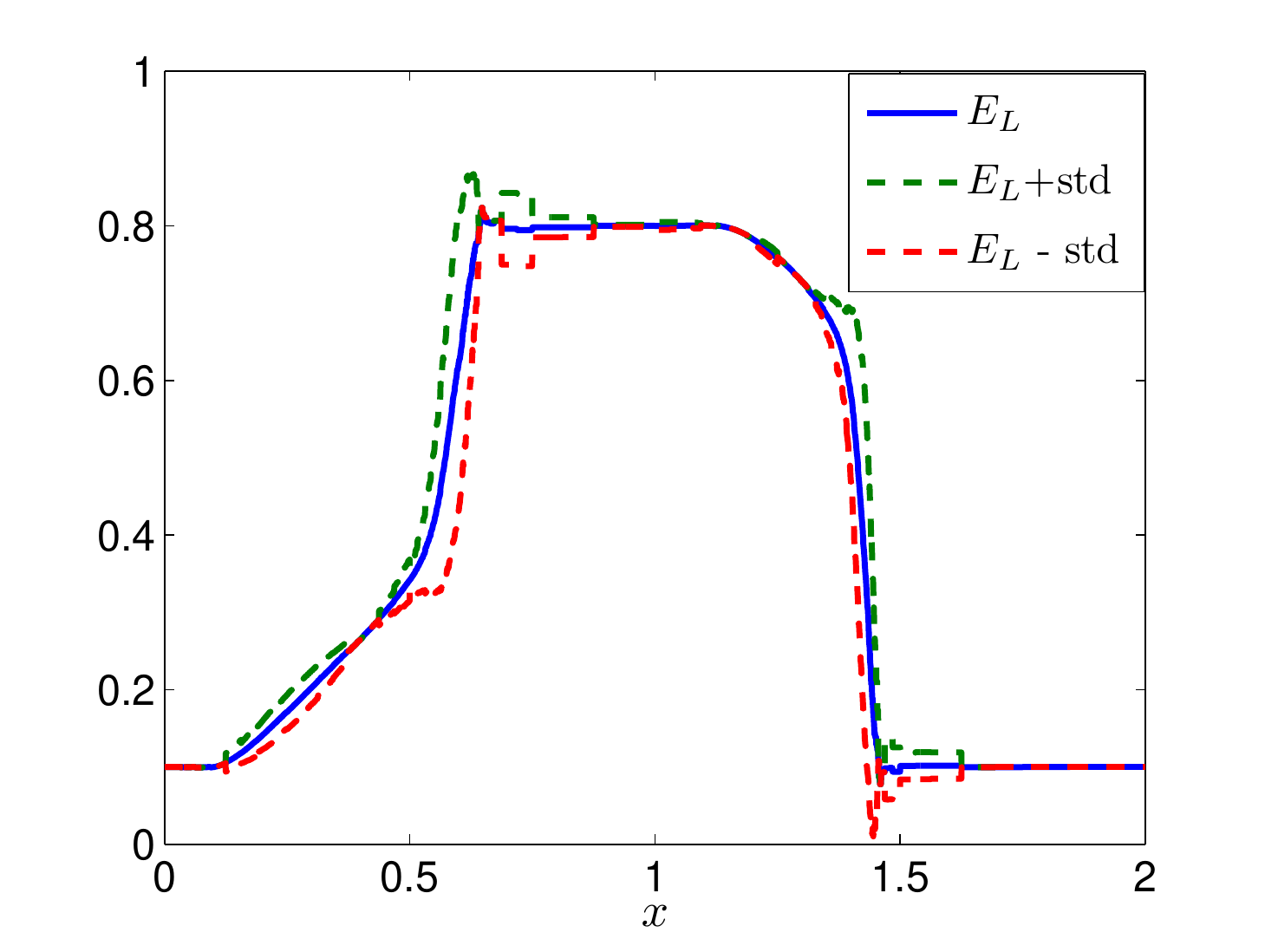}
  \end{tabular}
  \caption{Left: One sample of the random entropy solution of
    \eqref{eq:sw} with \eqref{eq:u02}, \eqref{eq:fa} and
    \eqref{eq:relpermexp} at time $T=0.3$ computed on a mesh with
    $4096$ points. Right: A sample of the estimator $E^L[s(\cdot,t)]$
    for \eqref{eq:sw} with \eqref{eq:u02}, \eqref{eq:fa} and
    \eqref{eq:relpermexp} at time $T=0.3$ (solid line), the dashed
    lines denote $E^L[s(\cdot,t)]$ $\pm$ standard deviation.}
  \label{fig:eg1}
\end{figure}
We will use this sample of the MLMC estimator as a reference solution
when estimating the approximation errors and computing the convergence
rates.

In order to compute an estimate on the error of the approximation of
the mean by the MLMC estimator $E^L[s(t)]$ in the $L^2(\Omega;
L^1(\R))$-norm, we use the relative error estimator introduced
in \cite{MishSch10a} based on a Monte Carlo quadrature in the
stochastic domain: By $U_{\mathrm{ref}}$ we denote a reference solution
and $\{U_k\}_{k=1,\dots,N}$ a sequence of independent approximate
solutions $E^L[s(t)]$ obtained by running the MLMC-FDM solver $N$
times, corresponding to $N$ realizations in the stochastic
domain. Then we estimate the relative error by
\begin{equation}\label{eq:errest1}
  \mathcal{R}E=\sqrt{\sum_{k=1}^N (\mathcal{R}E_k)^2/N},
\end{equation}
where
\begin{align*}
  \mathcal{R} E_k=100 \times
  \frac{\|U_{\mathrm{ref}}-U_k\|_{\ell^1}}{\|U_{\mathrm{ref}}\|_{\ell^1}}
\;.
\end{align*}
In \cite{MishSch10a}, the sensitivity of the error with respect to the
parameter $N$ is investigated.  In the present numerical experiments,
we use $N=5$ which was shown to be sufficient for most problems
\cite{MishSch10a, MishSchSuk12}.  In Table~\ref{tab:eg1} the errors
\eqref{eq:errest1} versus the resolution $\Delta x_L$ at the finest
level $L$ of the MLMC estimator and versus the average time (in
seconds) needed to compute one sample of the MLMC estimator are shown
($L=0,\dots,5$).  We observe that the calculated convergence rates are
$\approx 0.66$ (explicit scheme) and $\approx 0.75$ (implicit scheme)
with respect to the resolution and $\approx 0.32$ (explicit scheme)
and $0.38$ (implicit scheme) with respect to work. This is better than
what we would expect from the theory, cf. \eqref{4.49} and
\eqref{eq:mlmcwexp}, \eqref{eq:mlmcwimp}.  However, they decrease as
we refine the mesh, which might indicate that we are not in the
asymptotic regime yet.
\begin{table}[h]
  \centering
  \begin{tabular}[h]{|c|c|c|c|c|c|}
    \hline
    $L$ &  $\mathcal{R}E$ & $\Dx_L$ & run time &  $\abs{E^L(s(t))}_{BV([0,2])}$ & $\|E^L(s(t))\|_{L^{\infty}([0,2])}$ \\
    \hline
$0$ & 16.54 & 	$2^{-3}$  &  0.45  & 1.37       & 0.79\\
$1$ &  10.25  & $2^{-4}$  &  2.48  & 1.4	& 0.8\\
$2$ & 6.13 & 	$2^{-5}$  & 10.31  & 1.42	& 0.81\\
$3$ & 3.53 & 	$2^{-6}$  & 40.58  & 1.44	& 0.81 \\
$4$ & 2.06  & 	$2^{-7}$  &  159.6 & 1.48	& 0.81\\
$5$ &  1.68 & 	$2^{-8}$  &  632.37 &  1.58	& 0.82 \\
	\hline
average rate &  &   0.66  & -0.32 & &\\
    \hline
  \end{tabular}
  \caption{Relative mean square errors (as defined in \eqref{eq:errest1}) versus grid size at highest level and time (in seconds), for problem \eqref{eq:sw} with \eqref{eq:u02}, \eqref{eq:fa} and \eqref{eq:relpermexp}, for the MLMC solver with the explicit difference scheme.}
  \label{tab:eg1}
\end{table}
\begin{table}[h]
  \centering
  \begin{tabular}[h]{|c|c|c|c|c|c|}
 \hline
    $L$ &  $\mathcal{R}E$ & $\Dx_L$ & run time &  $\abs{E^L(s(t))}_{BV([0,2])}$ & $\|E^L(s(t))\|_{L^{\infty}([0,2])}$ \\
    \hline
$0$ & 22.89  & 	$2^{-3}$  &  0.65     & 1.3     & 0.76 \\ 
$1$ &  14.69 &  $2^{-4}$  &  3.53   &   1.39	&  0.8 \\
$2$ &  9.27  & 	$2^{-5}$  &  13.29 &    1.43	&  0.81\\
$3$ &  5.57  & 	$2^{-6}$  &   47.7 &  1.44	& 0.81 \\
$4$ &   3.15 & 	$2^{-7}$  & 174.8  &   1.47	& 0.81 \\
$5$ &  1.68 & 	$2^{-8}$  &  659.7  &  1.5	& 0.81 \\
	\hline
average rate & & 0.75 & -0.38 & & \\
\hline
  \end{tabular}
  \caption{Relative mean square errors (as defined in \eqref{eq:errest1}) 
           versus grid size at highest level and {CPU} time (in seconds), 
           for problem \eqref{eq:sw} with \eqref{eq:u02}, \eqref{eq:fa} and 
           \eqref{eq:relpermexp}, for the MLMC solver with \cs{time-stepping by}
           the implicit difference scheme.
          }
  \label{tab:eg1_2}
\end{table}
%
%
In the last two columns of Tables \ref{tab:eg1}, \ref{tab:eg1_2} the average total
variation and $L^\infty$-norm of $E^L(s(t))$ at the different
refinement levels are given. 
We observe that they slightly increase, 
but not as much as the bounds
\begin{align*}
  \abs{E^L(s(t))}_{BV(\R)}&=\biggl| \sum_{\ell=0}^L\frac{1}{M_\ell}
  \sum_{i=0}^{M_\ell}(s_\ell^i-s_{\ell-1}^i)\biggr|_{BV(\R)}
  \le  \sum_{\ell=0}^L\frac{1}{M_\ell}\sum_{i=0}^{M_\ell}
  \abs{s_\ell^i-s_{\ell-1}^i}_{BV(\R)}\\
  &\le 2 (L+1)\, \abs{s_0}_{BV(\R)}
\\
  \norm{E^L(s(t))}_{L^{\infty}(\mathbb{R})} 
  & = 
  \biggl\|
  \sum_{\ell=0}^L\frac{1}{M_\ell}\sum_{i=0}^{M_\ell}
  (s_\ell^i-s_{\ell-1}^i)\biggr\|_{L^\infty(\IR)}
  \le \sum_{\ell=0}^L\frac{1}{M_\ell}\sum_{i=0}^{M_\ell} 
  \norm{s_\ell^i-s_{\ell-1}^i}_{L^\infty(\R)}
\\
  &\le 2 (L+1)\, \norm{s_0}_{L^\infty(\R)},
\end{align*}
would imply.

\subsection{Random residual saturation}\label{ssec:residual}
In the following numerical example, we will model the relative
permeabilities by the random variables
\begin{align}\label{eq:relpermres}
  &\lambda^w(s)=\mathbf{1}_{s>s_{w}^*(\omega_1)} (s)
  \frac{\left(s-s_{w}^*(\omega_1)\right)^2}{\left(1-s_{w}^*(\omega_1)\right)^2},\quad
  \lambda^o(s)=\mathbf{1}_{s\le s_{o}^*(\omega_2)} (s)
  \left(1-\frac{s}{s_{o}^*(\omega_2)}\right)^2,\\ 
  &\notag\\
  &\text{with}\quad s^*_w(\omega_1)\sim \mathcal{U}(0.05,0.35),\quad
  s^*_o(\omega_2)\sim \mathcal{U}(0.6,0.95),\quad s^*_w(\omega_1)\perp
  s^*_o(\omega_2),\notag
\end{align}
that is we assume that the residual saturations $s^*_w$, $s^*_o$ are
independent, uniformly distributed random variables. As initial
data, we use again use \eqref{eq:u02} with periodic boundary
conditions.

The resulting $(s_0,f,A)(\omega_1,\omega_2;\cdot)$ again satisfies
assumptions \eqref{eq:rv} -- \eqref{eq:assAf}, so that the random
entropy solution from Definition \ref{def:res} exists and Theorems
\ref{theo3.2}, \ref{theo4.7} apply. In Figure \ref{fig:eg2} on the
left hand side, we have plotted a sample $s(\omega;t,\cdot)$ of the
random entropy solution at time $T=0.3$ and on the right hand side we
have plotted a sample of the MLMC-FDM estimator $E^L(s(t))$ for
$M_0=8$, $L=8$, $\Delta x_0=2^{-3}$, $K=1$ and CFL-number $0.4$. We
observe that the variance is larger compared to the variance in the
previous example.
\begin{figure}[ht]
  \begin{tabular}{lr}
    \includegraphics[width=0.49\textwidth]{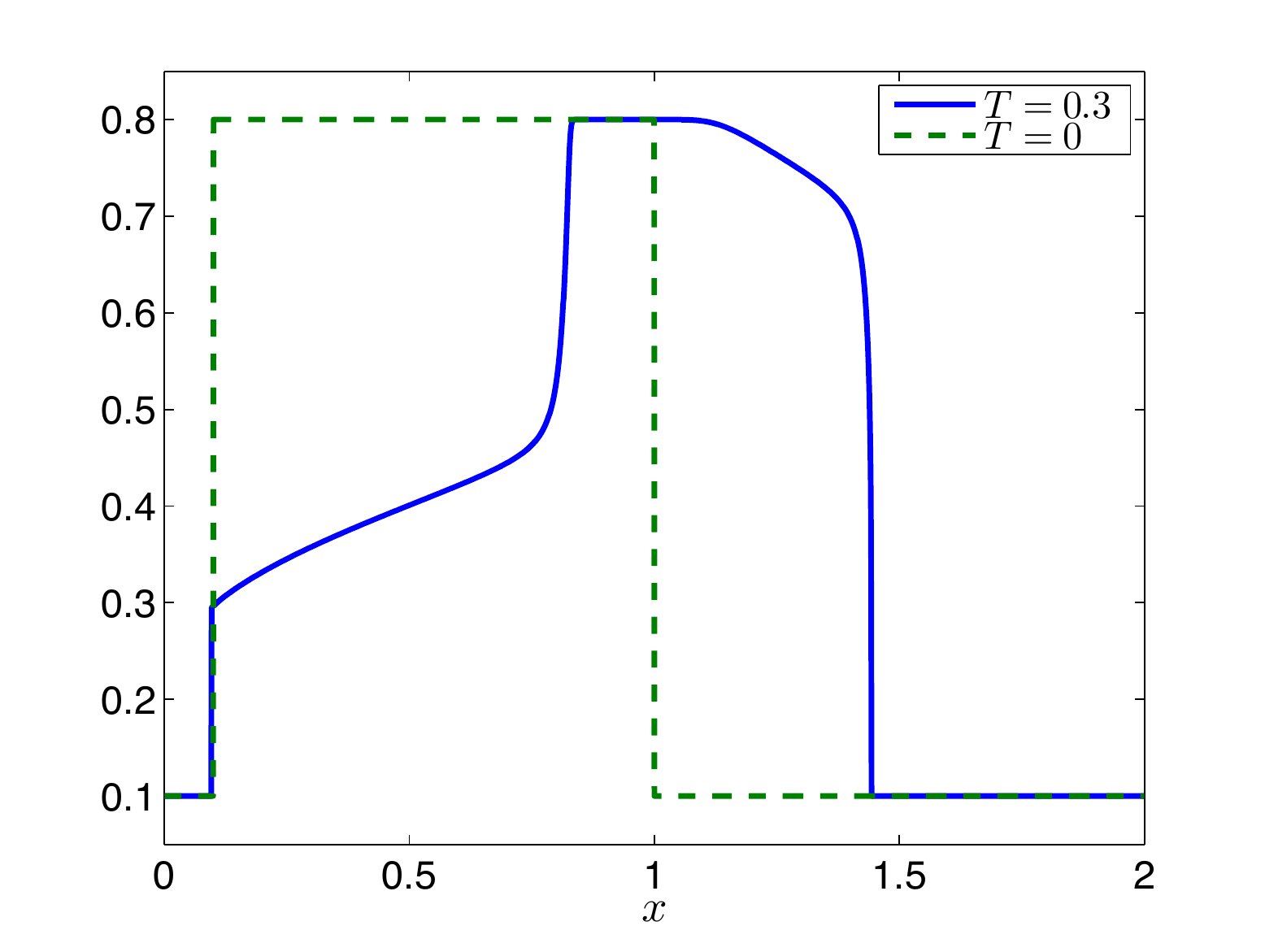}
    &\includegraphics[width=0.49\textwidth]{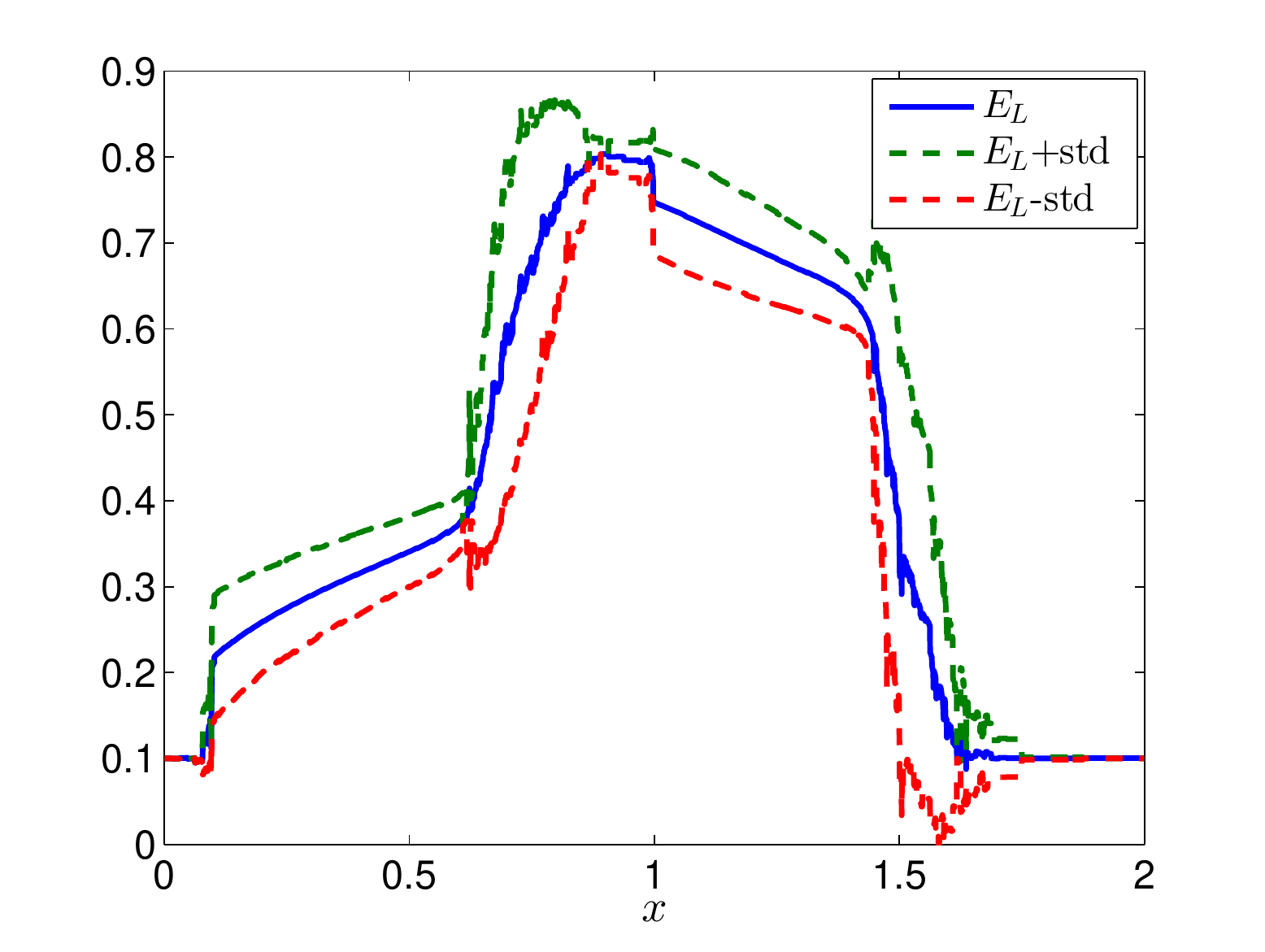}
  \end{tabular}
  \caption{Left: One sample of the random entropy solution of
    \eqref{eq:sw} with \eqref{eq:u02}, \eqref{eq:fa} and
    \eqref{eq:relpermres} at time $T=0.3$ computed on a mesh with
    $4096$ points. Right: A sample of the estimator $E^L[s(\cdot,t)]$
    for \eqref{eq:sw} with \eqref{eq:u02}, \eqref{eq:fa} and
    \eqref{eq:relpermres} at time $T=0.3$ (solid line), the dashed
    lines denote $E^L[s(\cdot,t)]$ $\pm$ standard deviation.}
  \label{fig:eg2}
\end{figure}
We will use this sample of the MLMC estimator as a reference solution when estimating the approximation errors and computing the convergence rates. Moreover, we will again compute an estimate of the $L^2(\Omega; L^1(\R))$-error using the error estimator defined in \eqref{eq:errest1} with $N=5$.

In Tables~\ref{tab:eg2}, \ref{tab:eg2_2}  the errors \eqref{eq:errest1} versus the
resolution $\Delta x_L$ at the finest level $L$ of the MLMC estimator
and versus the average time (in seconds) needed to compute one sample
of the MLMC estimator are shown ($L=0,\dots,5$). 
We observe that the
approximate convergence rates are $\approx 0.43$ (explicit scheme) and $\approx 0.57$ (implicit scheme) with respect to the
resolution and $\approx 0.21$ (explicit scheme) and $\approx 0.37$ (implicit scheme) with respect to work, which is again
better than what we would expect from the theory, cf. \eqref{4.49} and
\eqref{eq:mlmcwexp}, \eqref{eq:mlmcwimp}. However, it decreases as we refine the mesh,
which might indicate that we are not in the asymptotic regime yet. We
also note that the rates are lower than in the previous example.
\begin{table}[h]
  \centering
  \begin{tabular}[h]{|c|c|c|c|c|c|}
    \hline
    $L$ &  $\mathcal{R}E$ & $\Dx_L$ & run time & $\abs{E^L(s(t))}_{BV([0,2])}$ & $\|E^L(s(t))\|_{L^{\infty}([0,2])}$ \\
    \hline
$0$ & 12.36 & 	$2^{-3}$  &  0.48 & 1.31 & 0.76\\
$1$ & 8.67  & 	$2^{-4}$  &  2.69  & 1.38 & 0.78\\
$2$ & 5.44  & 	$2^{-5}$  &  10.82 & 1.47 & 0.82 \\
$3$ & 4.33  & 	$2^{-6}$  &  41.97 & 1.63 & 0.83 \\
$4$ & 3.32  & 	$2^{-7}$  &  164.53 & 1.73 & 0.82\\
$5$ & 2.77  & 	$2^{-8}$  &  679.27 & 2.04 & 0.82\\
	\hline
average rate &  &      0.43    & -0.21 & & \\
    \hline
  \end{tabular}
  \caption{Relative mean square errors (as defined in \eqref{eq:errest1}) versus grid size at highest level and time (in seconds), for problem \eqref{eq:sw} with \eqref{eq:u02}, \eqref{eq:fa} and \eqref{eq:relpermres} (explicit scheme in MLMC solver).}
  \label{tab:eg2}
\end{table}
%
\begin{table}[h]
  \centering
  \begin{tabular}[h]{|c|c|c|c|c|c|}
    \hline
    $L$ &  $\mathcal{R}E$ & $\Dx_L$ & run time & $\abs{E^L(s(t))}_{BV([0,2])}$ & $\|E^L(s(t))\|_{L^{\infty}([0,2])}$ \\
    \hline
$0$ & 16.0 & 	$2^{-3}$  &  1.16 &  1.23  &  0.72\\
$1$ & 9.34 & 	$2^{-4}$  &  3.9 &   1.37  &  0.78\\
$2$ & 5.54 & 	$2^{-5}$  &  11.52 &  1.44 &  0.8\\
$3$ & 4.54 & 	$2^{-6}$  &  31.29 &  1.61 &  0.82\\
$4$ & 2.61 & 	$2^{-7}$  &  88.8&    1.71 &  0.81\\
$5$ & 2.29 & 	$2^{-8}$  &  265.46 & 2.05 &  0.81\\
	\hline
average rate &  &      0.57    & -0.37  & & \\
    \hline
  \end{tabular}
  \caption{Relative mean square errors (as defined in \eqref{eq:errest1}) versus grid size at highest level and time (in seconds), for problem \eqref{eq:sw} with \eqref{eq:u02}, \eqref{eq:fa} and \eqref{eq:relpermres} (implicit scheme in MLMC solver).}
  \label{tab:eg2_2}
\end{table}

\subsection{Sine wave initial data}
As a third example, we have tested the convergence rates on Problem
\eqref{eq:sw}, \eqref{eq:fa}, \eqref{eq:relpermexp} with sine wave
initial data,
\begin{equation}\label{eq:sine}
s_0(x)=\sin(4\pi x),\quad x\in [0,0.5],
\end{equation}
In Figure~\ref{fig:eg3}, on the left hand side, we have plotted a
sample of the solution computed on a mesh with $2048$ points at time
$T=0.5$ and on the right hand side a sample of the MLMC estimator for
$L=7$, $\Dx_0=1/16$, $M_0=8$ and CFL-number $0.4$, also at time
$T=0.5$. 
We observe that the approximation looks quite smooth and that
the variance is evenly distributed over the whole spatial domain in
contrast to the previous examples.
\begin{figure}[h]
  \begin{tabular}{lr}
    \includegraphics[width=0.49\textwidth]{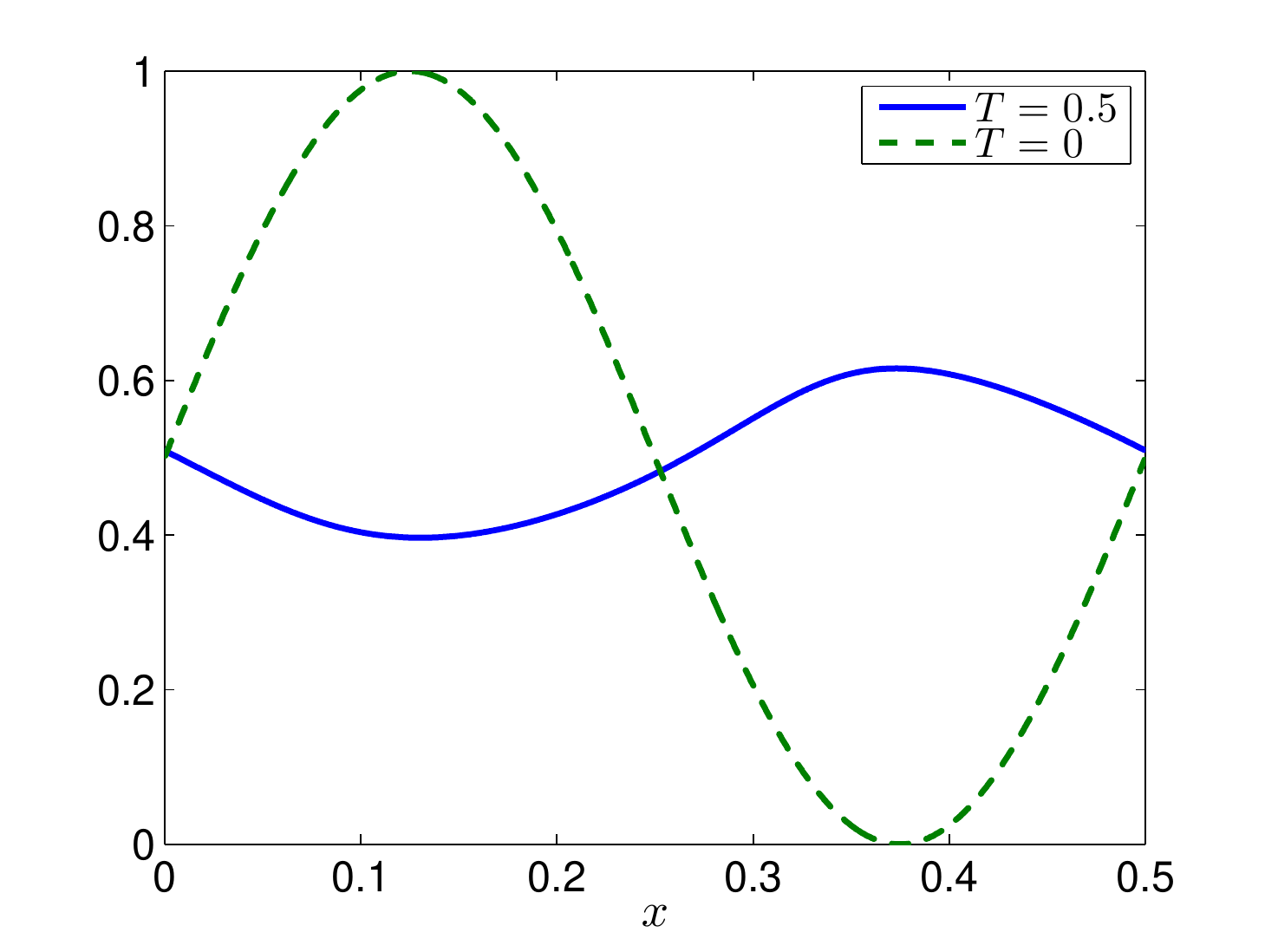}
    &\includegraphics[width=0.49\textwidth]{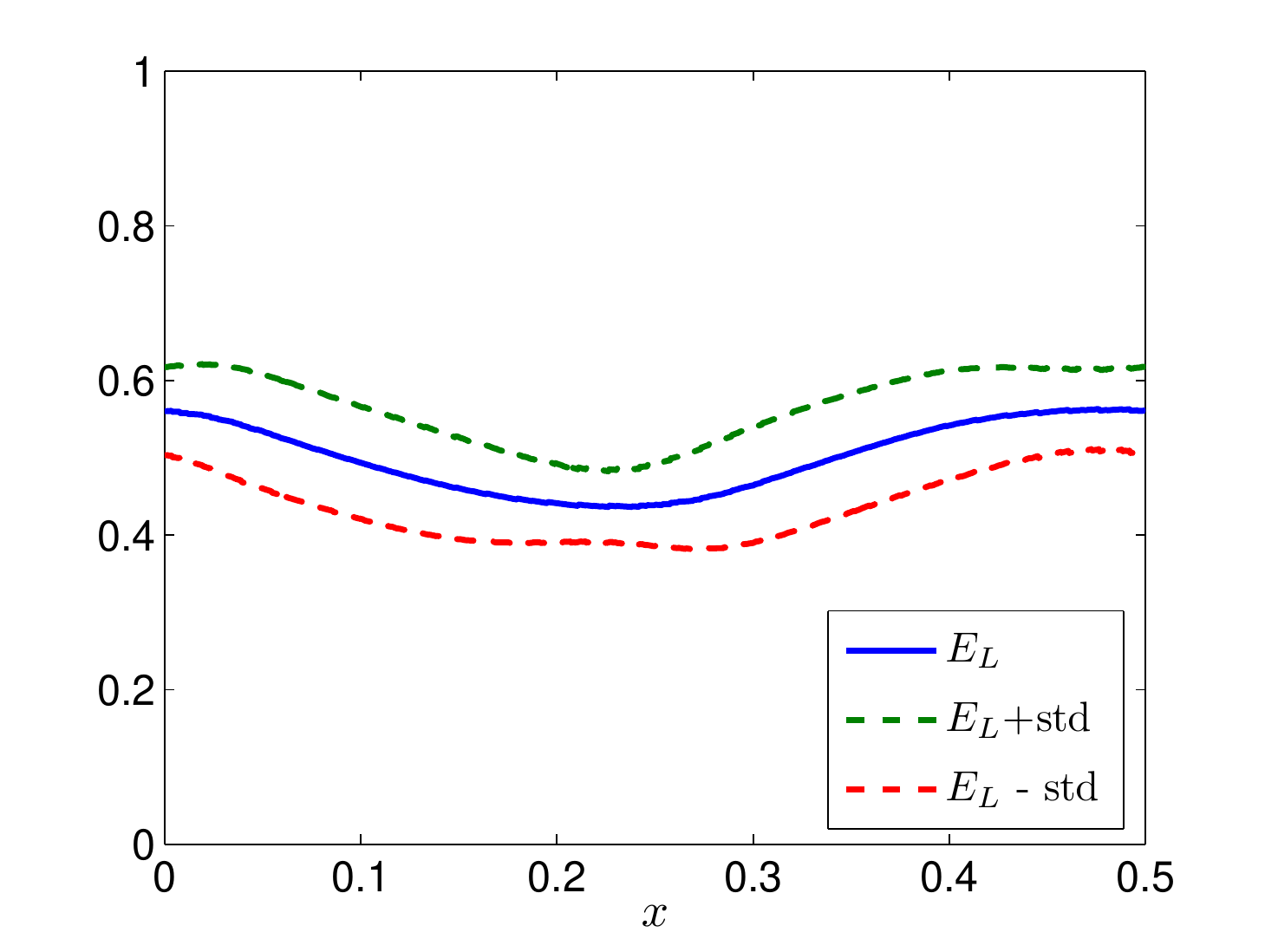}
  \end{tabular}
  \caption{Left: One sample of the random entropy solution of
    \eqref{eq:sw} with \eqref{eq:sine}, \eqref{eq:fa} and
    \eqref{eq:relpermexp} at time $T=0.5$ computed on a mesh with
    $2048$ points. Right: A sample of the estimator $E^L[s(\cdot,t)]$
    for \eqref{eq:sw} with \eqref{eq:sine}, \eqref{eq:fa} and
    \eqref{eq:relpermexp} at time $T=0.5$ (solid line), the dashed
    lines denote $E^L[s(\cdot,t)]$ $\pm$ standard deviation.}
  \label{fig:eg3}
\end{figure}
In Tables~\ref{tab:eg3}, \ref{tab:eg3_2} the estimates \eqref{eq:errest1} on the
$L^2(\Omega; L^1(\R))$-error for $N=5$ are displayed and compared to
the mesh resolution at the finest level and the run time.
\begin{table}[h]
  \centering
  \begin{tabular}[h]{|c|c|c|c|c|c|}
    \hline
    $L$ &  $\mathcal{R}E$ & $\Dx_L$ & run time  & $\abs{E^L(s(t))}_{BV([0,2])}$ & $\|E^L(s(t))\|_{L^{\infty}([0,2])}$\\
    \hline
$0$ & 8.11	& 	$2^{-4}$  &  1.68  & 0.008 	& 0.5\\
$1$ & 6.71      & 	$2^{-5}$  &  9.94  & 0.055	& 0.51\\
$2$ & 5.26 	& 	$2^{-6}$  &  43.36  & 0.131	& 0.53\\
$3$ & 3.35 	& 	$2^{-7}$  &  179.02 & 0.18	& 0.54\\
$4$ & 1.85      & 	$2^{-8}$  &  721.99 &	 0.253  & 0.56\\
	\hline
Average Rate &  &  0.53     & -0.25 & & \\
    \hline
  \end{tabular}
  \caption{Relative mean square errors (as defined in
    \eqref{eq:errest1}) versus grid size at highest level and time,
    for problem \eqref{eq:sw} with \eqref{eq:sine}, \eqref{eq:fa} and
    \eqref{eq:relpermexp} (explicit deterministic solver in MLMC method).} 
  \label{tab:eg3}
\end{table}
%
%
%

\begin{table}[h]
  \centering
  \begin{tabular}[h]{|c|c|c|c|c|c|}
    \hline
    $L$ &  $\mathcal{R}E$ & $\Dx_L$ & run time  & $\abs{E^L(s(t))}_{BV([0,2])}$ & $\|E^L(s(t))\|_{L^{\infty}([0,2])}$\\
    \hline
$0$ &  19.23     & 	$2^{-3}$  &  0.7    & 0.27	& 0.58 \\
$1$ &    5.71    & 	$2^{-4}$  &  3.94   & 0.13	& 0.54   \\
$2$ &    9.78	& 	$2^{-5}$  &   16.33 & 0.08	& 0.52   \\
$3$ &    6.6  	& 	$2^{-6}$  &  62.58  & 0.08	& 0.52   \\
$4$ &    4.91   & 	$2^{-7}$  &  235.8  & 0.14	& 0.53   \\
$5$ &   3.52   &        $2^{-8}$  & 898.15  & 0.22      & 0.55    \\   
	\hline
Average Rate &  &   0.39    & -0.19 & & \\
    \hline
  \end{tabular}
  \caption{Relative mean square errors (as defined in
    \eqref{eq:errest1}) versus grid size at highest level and time,
    for problem \eqref{eq:sw} with \eqref{eq:sine}, \eqref{eq:fa} and
    \eqref{eq:relpermexp} (implicit solver in MLMC method).} 
  \label{tab:eg3_2}
\end{table}

We find a convergence rate of $\approx 0.53$ (explicit scheme) and $\approx 0.39$ (implicit scheme) versus resolution and
$\approx 0.25$ (explicit scheme) and $\approx 0.19$ (implicit scheme) versus run time. We also observe that the rates improve
when the mesh is refined. Interestingly, for this example the convergence rates for the MLMC solver with the implicit scheme are worse than those for the explicit scheme in contrast to the previous examples. The reason could be the samples of the implicit scheme at level $L=1$ which are closer to the reference solution than the following ones at higher levels. This decreases the average rate for the implicit scheme. 
%
%
\bibliographystyle{abbrv}
\bibliography{stochbib}
\end{document}